\documentclass[12pt,a4paper]{amsart}  
\usepackage{graphicx,amssymb,stmaryrd,xspace}
\RequirePackage[raiselinks=false,colorlinks=true,citecolor=blue,urlcolor=black,linkcolor=blue,bookmarksopen=true]{hyperref}
\newcommand{\arxiv}[1]{\href{http://arxiv.org/pdf/#1}{arXiv:#1}}
\newcommand{\doi}[1]{\href{http://doi.org/#1}{DOI:#1}}

\usepackage{eucal}

\DeclareMathAlphabet{\mathbbe}{U}{bbold}{m}{n}
\newcommand{\simplexcategory}{\mathbbe{\Delta}}

\newcommand{\terminal}{1}
\def\Xiplus{\Xi_+}

\usepackage{texdraw}
\input{txdtools}

\everytexdraw{%
	\drawdim pt \linewd 0.5 \textref h:C v:C
	\setunitscale 1
}
\newcommand{\onedot}{
  \bsegment
    \move (0 0) \fcir f:0 r:2
  \esegment
}

\usepackage[v2,all]{xy}

\newcommand{\dlpullback}[1][dl]{\save*!/#1-4ex/#1:(-1,1)@^{|-}\restore}

\newcommand{\drpullback}[1][dr]{\save*!/#1-4ex/#1:(-1,1)@^{|-}\restore}

\usepackage{ stmaryrd }

\def\eq{{\mathrm{eq}}}

\def\Map{{\mathrm{Map}}}
\newcommand{\Eq}{\operatorname{Eq}}

\def\M{M\"obius\xspace}
\newcommand{\FILT}{tight\xspace}
\newdir{ >}{{}*!/-7pt/@{>}}
  
\def\overarrow#1{{\vec{#1}}}
\def\nondeg{\overarrow}

\newcommand{\Phieven}{\Phi_{\text{\rm even}}}
\newcommand{\Phiodd}{\Phi_{\text{\rm odd}}}

\newcommand{\wtil}{\widetilde}

\newcommand{\un}{\underline}
\newcommand{\Deltainj}{\simplexcategory_{\text{\rm inj}}}

\newcommand{\culf}{CULF\xspace}
\newcommand{\culfsymb}{\mathrm{culf}}

\renewcommand{\Im}{\operatorname{Im}}

\newcommand{\aInt}{\kat{aInt}}
\newcommand{\Int}{\kat{Int}}
\newcommand{\Decomp}{\kat{Dcmp}}
\newcommand{\cDecomp}{\kat{cDcmp}}
\newcommand{\FD}{\kat{FD}}
\newcommand{\cFD}{\kat{cFD}}

\def\Id{\text{Id}}

\newcommand{\Decbot}[1]{\operatorname{Dec}_\bot{}\kern-2pt{#1}}
\newcommand{\Dectop}[1]{\operatorname{Dec}_\top{}\kern-2pt{#1}}

\providecommand{\norm}[1]{\left| {#1}\right|}

\def\onto{\twoheadrightarrow}
\def\into{\hookrightarrow}

\providecommand{\kat}[1]{\text{\textbf{\textsl{#1}}}}
\newcommand{\LIN}{\kat{LIN}}

\newcommand{\HH}{\mathcal H}

\newcommand{\MI}{\ensuremath{M}}
\newcommand{\UU}{\ensuremath{U}}

\newcommand{\rat}{\rightarrowtail}

\newcommand{\upperstar}{^{\raisebox{-0.25ex}[0ex][0ex]{\(\ast\)}}}

\newcommand{\lowershriek}{_!}
\newcommand{\uppershriek}{^!}

\newcommand{\isopil}{\stackrel{\raisebox{0.1ex}[0ex][0ex]{\(\sim\)}}%
			{\raisebox{-0.15ex}[0.28ex]{\(\rightarrow\)}}}
\newcommand{\isleftadjointto}{\dashv}
\newcommand{\tensor}{\otimes}
\newcommand{\op}{^{\text{{\rm{op}}}}}

\newcommand{\Set}{\kat{Set}}
\newcommand{\Grpd}{\mathcal{S}}
\newcommand{\grpd}{\mathcal{F}}

\newcommand{\A}{\mathbb{A}}
\newcommand{\Q}{\mathbb{Q}}

\newcommand{\C}{\mathbb{C}}

\newcommand{\CC}{\mathcal{C}}
\newcommand{\DD}{\mathcal{D}}
\newcommand{\EE}{\mathcal{E}}

\newcommand{\comma}{\raisebox{1pt}{$\downarrow$}}
\newcommand{\name}[1]{\ulcorner #1\urcorner}

\newcommand{\Fun}{\operatorname{Fun}}
\newcommand{\id}{\operatorname{id}}
\newcommand{\colim}{\operatornamewithlimits{colim}}
\newcommand{\Ar}{\operatorname{Ar}}

\newcommand{\sker}{\kern0.5pt \mathrm{s}}

\newtheorem{lemma}{Lemma}[section]
\newtheorem{prop}[lemma]{Proposition}
\newtheorem{thm}[lemma]{Theorem}
\newtheorem{theorem}[lemma]{Theorem}
\newtheorem{cor}[lemma]{Corollary}
\theoremstyle{definition}

\newtheorem{BM}[lemma]{Remark}

\newtheorem{taller}[lemma]{$\!\!$}

\newenvironment{blanko}[1]%
{\begin{taller}{\normalfont\bfseries  #1}\normalfont}%
{\end{taller}}

\newenvironment{blanko*}[1]{\begin{list}{\bf {#1} }%
{\setlength{\labelsep}{0mm}\setlength{\leftmargin}{0mm}%
\setlength{\labelwidth}{0mm}\setlength{\listparindent}{\parindent}%
\setlength{\parsep}{\parskip}\setlength{\partopsep}{0mm}}%
\item%
}{\end{list}}

{\begin{list}{\em Definition. }%
{\setlength{\labelsep}{0mm}\setlength{\leftmargin}{0mm}%
\setlength{\labelwidth}{0mm}\setlength{\listparindent}{\parindent}%
\setlength{\parsep}{\parskip}\setlength{\partopsep}{0mm}}%
\item}{\end{list}}

\newenvironment{proof*}[1]{\begin{list}{\em #1 }%
{\setlength{\labelsep}{0mm}\setlength{\leftmargin}{0mm}%
\setlength{\labelwidth}{0mm}\setlength{\listparindent}{\parindent}%
\setlength{\parsep}{\parskip}\setlength{\partopsep}{0mm}}%
\item}{\qed\end{list}}

\thanks{%
  The first author 
  was partially supported by grants 
  MTM2012-38122-C03-01,  
  MTM2013-42178-P,       
  2014-SGR-634,          
  MTM2015-69135-P,   
  MTM2016-76453-C2-2-P  (AEI/FEDER, UE), and   
  2017-SGR-932,          
  the second author 
  by 
  MTM2013-42293-P, 
  MTM2016-80439-P  (AEI/FEDER, UE),
  and
  2017-SGR-1725,
  and the third author   
  by
  MTM2013-42178-P and
  MTM2016-76453-C2-2-P (AEI/FEDER, UE)}

\author{Imma G\'alvez-Carrillo}
\address{Departament de Matem\`atiques
      \\Universitat Polit\`ecnica de Catalunya
	  }
\email{m.immaculada.galvez@upc.edu}

\author{Joachim Kock}
\address{Departament de Matem\`atiques
       \\Universitat Aut\`onoma de Barcelona
	   }
\email{kock@mat.uab.cat}

\author{Andrew Tonks}
\address{Department of Mathematics\\ 
University of Leicester
}
\email{apt12@le.ac.uk}

\title[The decomposition space of M\"obius intervals]{Decomposition spaces, incidence algebras and M\"obius inversion
III: the decomposition space of M\"obius intervals}
\date{}         
\begin{document}
\begin{abstract}
  Decomposition spaces are simplicial $\infty$-groupoids subject to a certain
  exactness condition, needed to induce a coalgebra structure
  on the space of arrows.  Conservative ULF functors (\culf) between decomposition
  spaces induce coalgebra homomorphisms.  Suitable added finiteness
  conditions define the notion of \M decomposition space, a
  far-reaching generalisation of the notion of \M category of Leroux. 
  In this paper, we show that the
  Lawvere--Menni Hopf algebra of \M intervals,
  which contains the universal \M function (but
  is not induced by a \M
  category), can be realised as the homotopy cardinality of a
  \M decomposition space $\UU$ of all \M intervals, and that in a certain sense $\UU$ is universal
  for \M decomposition spaces and \culf functors.
\end{abstract}



\maketitle

\tableofcontents

\setcounter{section}{-1}

\addtocontents{toc}{\protect\setcounter{tocdepth}{1}}

\section*{Introduction}

This paper is the third of a trilogy dedicated to the study of
decomposition spaces and their incidence algebras.  

\medskip

In \cite{GKT:DSIAMI-1} we introduced the notion of decomposition space as a
general framework for incidence algebras and \M inversion.  
(Independently,
Dyckerhoff and Kapranov~\cite{Dyckerhoff-Kapranov:1212.3563},
motivated by geometry, representation theory and homological algebra,
had discovered the same notion, but formulated quite differently.)
A decomposition space is a simplicial $\infty$-groupoid $X$ satisfying
a certain exactness condition, weaker than the Segal condition.  Just as the Segal
condition expresses composition, the new condition expresses
decomposition, and there is an abundance of examples from combinatorics.
It is just the condition needed for a canonical coalgebra
structure to be induced on the slice $\infty$-category over $X_1$.  The comultiplication
is given by the span
$$
X_1 \stackrel{d_1}\longleftarrow X_2 \stackrel{(d_2,d_0)}\longrightarrow X_1\times X_1,
$$
which can be interpreted as saying that
comultiplying an edge $f\in X_1$ returns
the sum of all pairs of edges $(a,b)$ that are the short edges
of a triangle with long edge $f$.
If $X$ is the nerve of a category,
so
$f$ is an arrow,
then the $(a,b)$ are all pairs of arrows such that $b\circ a=f$.

In \cite{GKT:DSIAMI-2} we arrived at the notion of \M decomposition space, a
far-reaching generalisation of the notion of \M category of
Leroux~\cite{Leroux:1975}, by imposing suitable finiteness conditions on
decomposition spaces.  These notions will be recalled below.  

The present paper introduces the \M decomposition space of \M intervals,
subsuming discoveries made by Lawvere in the 1980s,
and establish that it is in a precise sense a universal
\M decomposition space.

\medskip

After Rota \cite{Rota:Moebius} and his collaborators \cite{JoniRotaMR544721} had
demonstrated the great utility of incidence algebras and \M inversion in locally
finite posets, and Cartier and Foata~\cite{Cartier-Foata} had developed a
similar theory for monoids with the finite-decomposition property, it was Leroux
who found the common generalisation, that of \M categories~\cite{Leroux:1975}.
These are categories with two finiteness conditions imposed: the first ensures
that an incidence coalgebra exists; the second ensures a general \M inversion
formula.
Conservative ULF functors (\culf) induce coalgebra 
homomorphisms~\cite{Content-Lemay-Leroux}, \cite{GKT:DSIAMI-1}.

Lawvere (in 1988, unpublished until Lawvere--Menni~\cite{LawvereMenniMR2720184})
observed that there is a universal coalgebra $\HH$ (in fact a Hopf algebra)
spanned by isomorphism classes of \M intervals.  From any incidence coalgebra of
a \M category there is a canonical coalgebra homomorphism to $\HH$, and the \M
inversion formula in the former is induced from a master inversion formula in
$\HH$.

Here is the idea: a {\em \M interval} is a \M category
\relax
with an initial object $0$ and terminal object $\terminal$
(not necessarily distinct).
The category of factorisations of any arrow $a$ in a \M category
$\C$ determines (\cite{Lawvere:statecats}) a \M interval $I(a)$
with $0$ given by the
factorisation $\id$-followed-by-$a$, and $\terminal$ by the
factorisation $a$-followed-by-$\id$.
There is a canonical \culf functor $I(a)\to\C$
sending $0\to \terminal$ to $a$, 
and since the arrow $0\to \terminal$ in $I(a)$ has
the same decomposition structure as the arrow $a$ in $\C$,
the comultiplication of $a$ can be calculated in $I(a)$.

Any collection of \M intervals that is closed under subintervals defines a 
coalgebra, and it
is an interesting integrability condition for such a collection to come from a
single \M category.  
The Lawvere--Menni coalgebra is simply the collection of {\em
all} isomorphism classes of \M intervals.

Now, the coalgebra of \M intervals cannot be the coalgebra of a single Segal
space, because such a Segal space $\UU$ would have to have $\UU_1$ the space
of all \M intervals, and $\UU_2$ the space of all subdivided \M intervals.  But a
\M interval with a subdivision (i.e.~a `midpoint') contains more information than
the two parts of the subdivision: one from $0$ to the midpoint, and one from the
midpoint to $\terminal$:

\begin{center}\begin{texdraw}

  \setunitscale 1
  \arrowheadtype t:V
  \arrowheadsize l:5 w:4
  \move (0 0) 
  \bsegment
  \move (0 0) \clvec (12 10)(28 10)(40 0) 
  \move (0 0) \clvec (12 -10)(28 -10)(40 0)
  \move (0 0) \onedot
  \move (20 0) \onedot
  \move (40 0) \onedot
  \esegment

  \move (70 0)\htext{$\neq$}
  \move (80 0)
  \bsegment
  \move (20 0)
  \clvec (26 7)(34 7)(40 0)
  \move (20 0)
  \clvec (26 -7)(34 -7)(40 0)
  \move (40 0)
  \clvec (46 7)(54 7)(60 0)
  \move (40 0)
  \clvec (46 -7)(54 -7)(60 0)
  \move (20 0) \onedot
  \move (40 0) \onedot
  \move (60 0) \onedot
  \move(60 12)
  \move(60 -12)
  \esegment
\end{texdraw}\end{center}
This is to say that the Segal condition is not satisfied: we have
\begin{eqnarray*}
  \UU_2 & \neq & \UU_1 \times_{\UU_0} \UU_1 .
\end{eqnarray*}

We shall prove that the simplicial space of all intervals and their subdivisions
{\em is} a decomposition space, as suggested by this figure:

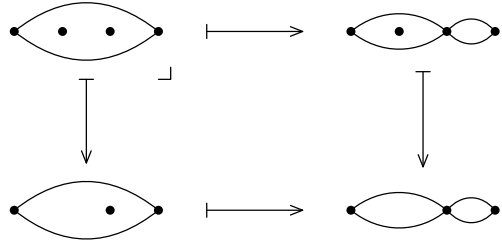
\begin{center}\begin{texdraw}

  \setunitscale 0.9
  \arrowheadtype t:V
  \arrowheadsize l:5 w:4
  \move (0 90)
  \move (0 75) 
  \bsegment
  \move (0 0) \clvec (20 16)(40 16)(60 0)
  \move (0 0) \clvec (20 -16)(40 -16)(60 0)
  \move (0 0) \onedot
  \move (20 0) \onedot
  \move (40 0) \onedot
  \move (60 0) \onedot
  \esegment
  
  \move (80 -3) \rlvec (0 6)
  \move (80 0) \ravec (40 0)
  \move (80 72) \rlvec (0 6)
  \move (80 75) \ravec (40 0)
  \move (27 55) \rlvec (6 0)
  \move (30 55) \ravec (0 -35)
  \move (167 58) \rlvec (6 0)
  \move (170 58) \ravec (0 -40)
  
  \move ( 60 55) \rlvec (5 0) \rlvec (0 5)
  \move (0 0)
  \bsegment
  \move (0 0) \clvec (20 16)(40 16)(60 0)
  \move (0 0) \clvec (20 -16)(40 -16)(60 0)
  \move (0 0) \onedot
  \move (40 0) \onedot
  \move (60 0) \onedot
  \esegment
  
  \move (140 75) 
  \bsegment
  \move (0 0) \clvec (12 10)(28 10)(40 0) \clvec (46 7)(54 7)(60 0)
  \move (0 0) \clvec (12 -10)(28 -10)(40 0) \clvec (46 -7)(54 -7)(60 0)
  \move (0 0) \onedot
  \move (20 0) \onedot
  \move (40 0) \onedot
  \move (60 0) \onedot
  \esegment

    \move (140 0) 
  \bsegment
  \move (0 0) \clvec (12 10)(28 10)(40 0) \clvec (46 7)(54 7)(60 0)
  \move (0 0) \clvec (12 -10)(28 -10)(40 0) \clvec (46 -7)(54 -7)(60 0)
  \move (0 0) \onedot
  \move (40 0) \onedot
  \move (60 0) \onedot
  \esegment\move (0 -12)
\end{texdraw}\end{center}
meant to indicate that this diagram is a pullback:
$$\xymatrix @C=44pt @R=30pt{
   \UU_3 \drpullback \ar[r]^-{(d_3,d_0d_0)}\ar[d]_{d_1} & \UU_2\times_{\UU_0} \UU_1 \ar[d]^{d_1 \times 
   \id} \\
   \UU_2 \ar[r]_-{(d_2,d_0)} & \UU_1 \times_{\UU_0} \UU_1
}$$
which in turn is one of the conditions involved in the decomposition-space 
axiom. 

While the ideas outlined have a clear intuitive content, a considerable amount of
machinery is needed actually to construct the universal decomposition space, and to get
sufficient hold of its structural properties to prove the desired results about
it.  We first work out the theory without finiteness conditions, which we impose
at the end.

\medskip

Let us outline our results in more detail.

First of all we need to develop a theory of intervals in the framework of
decomposition spaces.  Lawvere's idea \cite{Lawvere:statecats} is that
to an arrow one may associate its category of factorisations, which is an
interval.  To set this up, we exploit factorisation systems and adjunctions
derived from them, and start out in Section~\ref{sec:fact} with some general
results about factorisation systems, some results of which are already available
in Lurie's book~\cite{Lurie:HTT}.  Specifically we describe a situation in which
a factorisation system lifts across an adjunction to produce a new factorisation
system, and hence a new adjunction.

Before coming to intervals in Section~\ref{sec:fact-int}, we need flanked decomposition
spaces (Section~\ref{sec:flanked}): these are certain presheaves on the category $\Xi$
of nonempty finite linear orders with a top and a bottom element.  The
$\infty$-category of flanked decomposition spaces features the important {\em
stretched-cartesian} factorisation system, where `stretched' is to be thought of as
endpoint-preserving, and cartesian is like `distance-preserving'.  There is also
the basic adjunction between decomposition spaces and flanked decomposition
spaces, which in fact is the double decalage construction (this is interesting
since decalage already plays an important part in the theory of decomposition 
spaces~\cite{GKT:DSIAMI-1}). Intervals are first
defined as certain flanked decomposition spaces which are contractible in degree
$-1$ (this condition encodes an initial and a terminal object) (\ref{aInt}), and
via the basic adjunction we obtain the definitive $\infty$-category of intervals
as a full subcategory of the $\infty$-category of complete decomposition spaces
(\ref{towards-Int}); it features the stretched-\culf factorisation system
(\ref{fact-Int}), which extends the active-inert
(a.k.a.~generic-free) factorisation system on
$\simplexcategory$ (\ref{IntDelta}).  The factorisation-interval construction can now
finally be described (Theorem~\ref{Thm:I}) as a coreflection from complete
decomposition spaces to intervals (or more precisely, on certain coslice
$\infty$-categories).  We show that every interval is a Segal space
(\ref{prop:i*flanked=Segal}).  The simplicial space $\UU$ of
intervals (which lives in a bigger universe) can finally (\ref{U}) be defined
very formally in terms of a natural right fibration over $\simplexcategory$ whose total space has
objects stretched interval maps from an ordinal.  In plain words, $\UU$ consists of
subdivided intervals.

\medskip

With these various preliminary technical constructions having taken up
two thirds of the paper, we can finally state and prove the main results:

\medskip

\noindent
{\bf Theorem~\ref{UcompleteDecomp}.} {\em $\UU$ is a complete decomposition space.}

\medskip

\noindent
The factorisation-interval construction yields, for every decomposition space $X$, a canonical functor $X \to \UU$,
called the {\em classifying map}.

\medskip

\noindent
{\bf Theorem~\ref{thm:IU=cULF}.} {\em The classifying map is \culf.}

\medskip

We conjecture that $\UU$ is universal for complete 
decomposition spaces and \culf maps, and prove the following partial result:

 \medskip
 
 \noindent
{\bf Theorem~\ref{thm:connected}.} {\em 
For each complete decomposition space $X$, the space 
  $\Map_{\cDecomp^{\culfsymb}}(X,\UU)$ is connected.}
  
  \medskip

We finish in Section~\ref{sec:MI} by imposing the \M condition, obtaining
the corresponding finite results.
A {\em \M interval} is an interval which is \M as a decomposition space.
We show that every \M interval is a Rezk complete Segal space 
(\ref{prop:Mint=Rezk}).
There is a decomposition space of {\em all}\/ \M intervals, and it is shown to be small.

Our final theorem is now:

 \medskip
 
 \noindent
{\bf Theorem~\ref{thm:MI=M}.} {\em The decomposition space of all \M intervals 
is \M.}

\medskip

\noindent
It follows that it admits a \M inversion formula with coefficients in finite 
$\infty$-groupoids
or in $\Q$, and since every \M decomposition space admits a 
canonical \culf functor to it, we find that \M inversion in every incidence 
algebra (of a \M decomposition space) is induced from this master formula.

\bigskip

\noindent
{\bf Note.}
This work was originally Section 7 of a large single
manuscript 
{\em Decomposition spaces, incidence algebras 
  and \M inversion} \cite{GKT:1404.3202}.
For publication, this manuscript has been split into six papers:
  \begin{enumerate}
    \item[(0)]  Homotopy linear algebra \cite{GKT:HLA}
    
    \item  Decomposition spaces, incidence algebras and \M inversion I: basic 
    theory \cite{GKT:DSIAMI-1}
  
    \item  Decomposition spaces, incidence algebras and \M inversion II:
    completeness, length filtration, and finiteness \cite{GKT:DSIAMI-2}
  
    \item  Decomposition spaces, incidence algebras and \M inversion III:
    the decomposition space of \M intervals [present paper]
  
    \item  Decomposition spaces and restriction species \cite{GKT:restriction}
  
    \item  Decomposition spaces in combinatorics \cite{GKT:ex}
    
  \end{enumerate}
  
\bigskip

\noindent
{\bf Acknowledgments.}
We wish to thank Andr\'e Joyal and David Gepner
for many enlightening discussions that helped shape
this work, and the referee for comments that have improved the exposition.

\section{Decomposition spaces}

We briefly recall from \cite{GKT:DSIAMI-1} the notions of decomposition
space and \culf functors, and a few key results needed.

\begin{blanko}{The setting: $\infty$-categories.}
  We work in the $\infty$-category of $\infty$-categories, and refer to
  Lurie's {\em Higher Topos Theory}~\cite{Lurie:HTT} for background.
  Thanks to the monumental effort of Joyal~\cite{Joyal:qCat+Kan},
  \cite{Joyal:CRM} and Lurie~\cite{Lurie:HTT}, it is now possible to work
  model-independently, at least as long as the category theory involved is
  not too sophisticated.  This is the case in the present work, where most
  of the constructions are combinatorial, dealing as they do with various
  configurations of $\infty$-groupoids, and it is feasible to read most of
  the paper substituting the word set for the word $\infty$-groupoid.  In
  fact, even at that level of generality, the results are new and
  interesting.

  Working model-independently has a slightly different flavour than many of
  the arguments in the works of Joyal and Lurie, who, in order to bootstrap
  the theory and establish all the theorems we now harness, had to work in
  the category of simplicial sets with the Joyal model structure.  For
  example, throughout when we refer to a slice $\infty$-category $\CC_{/X}$
  (for $X$ an object of an $\infty$-category $\CC$), we only refer to an
  $\infty$-category determined up to equivalence of $\infty$-categories by
  a certain universal property (Joyal's insight of defining slice
  categories as adjoint to a join operation~\cite{Joyal:qCat+Kan}).  In the
  Joyal model structure for quasi-categories, this category can be
  represented by an explicit simplicial set.  However, there is more than
  one possibility, depending on which explicit version of the join operator
  is employed (and of course these are canonically equivalent).  In the
  works of Joyal and Lurie, these different versions are distinguished, and
  each has some technical advantages.  In the present work we shall only
  need properties that hold for both, and we shall not distinguish between
  them.
\end{blanko}

\begin{blanko}{Linear algebra with coefficients in $\infty$-groupoids \cite{GKT:HLA}.}
  Let $\Grpd$ denote the $\infty$-category of $\infty$-groupoids.
  The slice $\infty$-categories $\Grpd_{/I}$ form the objects
  of a symmetric monoidal $\infty$-category $\LIN$, described in
  detail in \cite{GKT:HLA}: the morphisms are the linear functors,
  meaning that they preserve homotopy sums, or equivalently indeed
  all colimits.  Such functors are given by spans: the span
  $$
  I \stackrel p \leftarrow M \stackrel q \to J$$
  defines the linear functor
  $$
  q\lowershriek \circ p\upperstar : \Grpd_{/I}  \longrightarrow  \Grpd_{/J} 
  $$
  given by pullback along $p$ followed by composition with $q$.
  The $\infty$-category $\LIN$ can play the role of the category of
  vector spaces, although to be strict about that interpretation,
  finiteness conditions should be imposed, as we do later in this paper 
  (Section~\ref{sec:MI}).
  The symmetric monoidal structure on $\LIN$ is given on objects by
  $$
  \Grpd_{/I} \tensor \Grpd_{/J} = \Grpd_{I\times J} ,
  $$
  just as the tensor product of vector spaces with bases indexed by sets $I$ 
  and $J$ is the vector space with basis indexed by $I\times J$.
  The neutral object is $\Grpd$.
\end{blanko}

\begin{blanko}{Active and inert maps (generic and free maps).}\label{generic-and-free}
  The category $\simplexcategory$ of nonempty finite ordinals and monotone maps
  has an active-inert factorisation system:
  an arrow $a: [m]\to [n]$ in $\simplexcategory$ is \emph{active} 
    (also called {\em generic})
  when it preserves end-points, $a(0)=0$ and
  $a(m)=n$; and it is \emph{inert} 
  (also called {\em free})
  if it is distance preserving,
  $a(i+1)=a(i)+1$ for $0\leq i\leq
  m-1$.  The active maps are generated by the codegeneracy maps and
  the inner coface maps, while the inert maps are generated by the outer
  coface maps.  Every morphism in $\simplexcategory$ factors uniquely as an
  active
  map followed by an inert map.

  (The notions of generic and free maps are general notions in category theory,
  introduced by Weber \cite{Weber:TAC13,Weber:TAC18}, who extracted the
  notions from earlier work of Joyal~\cite{Joyal:foncteurs-analytiques}; a
  recommended entry point to the theory is
  Berger--Melli\`es--Weber~\cite{Berger-Mellies-Weber:1101.3064}.
  The more recent terminology `active/inert' is due to 
  Lurie~\cite{Lurie:HA}, and is more suggestive for the role the two 
  classes of maps play.)
\end{blanko}

\begin{lemma}\label{genfreepushout}
  Active and inert maps in $\simplexcategory$ admit pushouts along each other, and the 
  resulting maps are again active and inert.
\end{lemma}

\begin{blanko}{Decomposition spaces \cite{GKT:DSIAMI-1}.}
  A simplicial space $X:\simplexcategory\op\to\Grpd$ is called a {\em decomposition space}
  when it takes active-inert pushouts in $\simplexcategory$ to pullbacks.

  Every Segal space is a decomposition space.  The main construction in 
  the present paper, the decomposition space of intervals, is an example
  which is not a Segal space.
\end{blanko}

The notion of decomposition space can be seen as an abstraction of coalgebra: it
is precisely the condition required to obtain a counital coassociative
comultiplication on $\Grpd_{/X_1}$.  The following is the main 
theorem of \cite{GKT:DSIAMI-1}.
\begin{theorem}\label{thm:comultcoass} \cite{GKT:DSIAMI-1}
  For $X$ a decomposition space, the slice $\infty$-category $\Grpd_{/X_1}$ has
  the structure of a strong homotopy comonoid in the symmetric monoidal
  $\infty$-category $\LIN$, with the comultiplication defined by the span
  $$
  X_1 \stackrel{d_1}\longleftarrow X_2 \stackrel{(d_2,d_0)}
  \longrightarrow X_1 \times X_1 . 
  $$
\end{theorem}
If $X$ is the nerve of a locally finite category (for example a poset),
then (the cardinality of) this comultiplication is that of the classical 
incidence coalgebra,
$$
\Delta(f) = \sum_{b \circ a=f} a \tensor b .
$$

\begin{blanko}{\culf maps.}
  For the present purposes, the relevant notion of morphism is that of conservative ULF map: A
  simplicial map is called {\em ULF} (unique lifting of factorisations) if
  it is cartesian on active face maps, and it is called {\em conservative}
  if cartesian on degeneracy maps.  We write {\em \culf} for conservative
  and ULF, that is, cartesian on all active maps.
  
  The \culf maps induce coalgebra homomorphisms.
\end{blanko}

\begin{blanko}{Decalage.}
  (See Illusie~\cite{Illusie2}).  Given a simplicial space $X$ as in the top
  row of the following diagram, the {\em lower dec} $\Decbot{X}$ is a new
  simplicial space (the bottom row of the diagram) obtained by deleting
  $X_0$ and shifting everything one place down, deleting also all $d_0$
  face maps and all $s_0$ degeneracy maps.  It comes equipped with a
  simplicial map, the {\em dec map}, $d_\bot:\Decbot{X}\to X$
  given by the original $d_0$:

$$
\xymatrix@C+1em{
X_0  
\ar[r]|(0.55){s_0} 
&
\ar[l]<+2mm>^{d_0}\ar[l]<-2mm>_{d_1} 
X_1  
\ar[r]<-2mm>|(0.6){s_0}\ar[r]<+2mm>|(0.6){s_1}  
&
\ar[l]<+4mm>^(0.6){d_0}\ar[l]|(0.6){d_1}\ar[l]<-4mm>_(0.6){d_2}
X_2 
\ar[r]<-4mm>|(0.6){s_0}\ar[r]|(0.6){s_1}\ar[r]<+4mm>|(0.6){s_2}  
&
\ar[l]<+6mm>^(0.6){d_0}\ar[l]<+2mm>|(0.6){d_1}\ar[l]<-2mm>|(0.6){d_2}\ar[l]<-6mm>_(0.6){d_3}
X_3 
\ar@{}|\cdots[r]
&
\\
\\
X_1  \ar[uu]_{d_0}
\ar[r]|(0.55){s_1} 
&
\ar[l]<+2mm>^{d_1}\ar[l]<-2mm>_{d_2} 
X_2  \ar[uu]_{d_0}
\ar[r]<-2mm>|(0.6){s_1}\ar[r]<+2mm>|(0.6){s_2}  
&
\ar[l]<+4mm>^(0.6){d_1}\ar[l]|(0.6){d_2}\ar[l]<-4mm>_(0.6){d_3}
X_3 \ar[uu]_{d_0}
\ar[r]<-4mm>|(0.6){s_1}\ar[r]|(0.6){s_2}\ar[r]<+4mm>|(0.6){s_3}  
&
\ar[l]<+6mm>^(0.6){d_1}\ar[l]<+2mm>|(0.6){d_2}\ar[l]<-2mm>|(0.6){d_3}\ar[l]<-6mm>_(0.6){d_4}
X_4 \ar[uu]_{d_0}
\ar@{}|\cdots[r]
&
}
$$

Similarly, the upper dec, denoted $\Dectop{X}$ is obtained by instead 
deleting, in each degree, the last face map $d_\top$ and the last degeneracy map 
$s_\top$.

  The functor $\Decbot{}$ can be described more conceptually as follows (see
  Lawvere~\cite{Lawvere:ordinal}).  There is an `add-bottom' endofunctor
  $b:\simplexcategory\to \simplexcategory$, which sends $[k]$ to $[k+1]$ by adding a new bottom
  element.  This is in fact a monad; the unit $\varepsilon: \Id \Rightarrow b$ is
  given by the bottom coface map $d^\bot$.  The lower dec is given by
  precomposition with $b$:
  $$
  \Decbot{X} = b\upperstar X
  $$
  Hence $\Decbot{}$ is a comonad, and its counit is the bottom face map $d_\bot$.

  Similarly, the upper dec is obtained from the `add-top' monad on $\simplexcategory$. 
  Below we shall exploit crucially the combination of the two comonads.
  
  The following result from \cite[Theorem 4.10]{GKT:DSIAMI-1} will be 
  invoked several times:
\end{blanko}

\begin{thm}\label{thm:decomp-dec-segal}
  $X$ is a decomposition space if and only if $\Dectop{X}$ and $\Decbot{X}$ are
  Segal spaces, and the dec maps $d_\top : \Dectop{X} \to X$ and
  $d_\bot :
  \Decbot{X} \to X$ are \culf.
\end{thm}

\begin{blanko}{Complete decomposition spaces \cite{GKT:DSIAMI-2}.} 
  A decomposition space $X: \simplexcategory\op\to\Grpd$ is {\em complete}
  when $s_0: X_0 \to X_1$ is a monomorphism (i.e.~is $(-1)$-truncated).
  It follows from the decomposition space axiom that in this case
  {\em all} degeneracy maps are monomorphisms.
  
  A Rezk complete Segal space is a complete decomposition space.
  The motivation for the completeness notion is to get a good notion
  of nondegenerate simplices, in turn needed to obtain the \M inversion
  principle.  The completeness condition is also needed to formulate
  the `tightness' condition, locally finite length, which we come to
  in \ref{bla:tight} below.
\end{blanko}

\section{Factorisation systems and cartesian fibrations}
\label{sec:fact}

In this section, which makes no reference to decomposition spaces, we prove some general
results in category theory to the effect of lifting factorisation systems
along an adjunction, and the like.
For background to this section, see Lurie~\cite[\S~5.2.8]{Lurie:HTT}.

\begin{blanko}{Factorisation systems.}
  A {\em factorisation system} on an $\infty$-category $\DD$ consists of two classes $E$ and $F$
  of maps, that we shall depict as $\onto$ and $\rat$, such that
  \begin{enumerate}
    \item The classes $E$ and $F$ are closed under equivalences.
  
    \item The classes $E$ and $F$ are orthogonal, $E\bot F$.  That is, given $e\in E$ and $f\in F$, for every solid square
    $$\xymatrix{
       \cdot \ar[r]\ar@{->>}[d]_e & \cdot \ar@{ >->}[d]^f \\
       \cdot \ar@{-->}[ru]\ar[r] & \cdot
    }$$
    the space of fillers is contractible.

    \item Every map $h$ admits a factorisation
  $$
  \xymatrix@R-1em@C+1em{\cdot\ar[rr]^h\ar@{->>}[rd]_e&&\cdot\\&\cdot\ar@{ >->}[ru]_f}
  $$
  with $e\in E$ and $f\in F$.
  \end{enumerate}

  (Note that in \cite[Definition 5.2.8.8]{Lurie:HTT},
  the first condition is given as `stability under formation of retracts'.  In 
  fact this stability follows from the three conditions above.  Indeed, suppose 
  $h \bot F$; 
  factor $h = f \circ e$ as above.  Since $h\bot f$, there is a diagonal filler in
  $$\xymatrix{
     \cdot \ar@{->>}[r]^e\ar[d]_h & \cdot \ar@{ >->}[d]^f \\
     \cdot \ar@{-->}[ru]^-d \ar[r]_{\id} & \cdot
  }$$
  Now $d$ belongs to ${}^\bot F$ since $e$ and $h$ do, and $d$ belongs to
  $E^\bot$ since $f$ and $\id$ do.  Hence $d$ is an equivalence, and therefore
  $h \in E$, by equivalence stability of $E$.  Hence $E= {}^\bot F$,
  and is therefore closed under retracts.  Similarly for $F$.  It also follows 
  that the two classes are closed under composition.)
\end{blanko}

\begin{blanko}{Set-up.}\label{setup}
  In this section, fix an $\infty$-category $\DD$ with a factorisation system
  $(E,F)$ as above.  
Let $\Ar(\DD)= \Fun(\Delta[1],\DD)$, 
 whose $0$-simplices we depict vertically, then the domain projection $\Ar(\DD) 
  \to\DD$ (induced by the inclusion $\{0\} \into \Delta[1]$) is a cartesian 
  fibration; the cartesian arrows are the squares of the form
  $$\xymatrix{
  \cdot \ar[r] \ar[d] & \cdot \ar[d] \\
  \cdot 
  \ar[r]^\sim & \cdot}
  $$

  Let $\Ar^E(\DD) \subset \Ar(\DD)$ denote the full subcategory spanned by the
  arrows in the left-hand class $E$.
\end{blanko}

\begin{lemma}\label{ArED}
  The domain projection $\Ar^E(\DD) \to \DD$ is a cartesian fibration. The cartesian
  arrows in $\Ar^E(\DD)$ are given by squares of the form
  $$\xymatrix{
  \cdot \ar[r] \ar@{->>}[d] & \cdot \ar@{->>}[d] \\
  \cdot 
  \ar@{ >->}[r] & \cdot}$$
\end{lemma}
\begin{proof}
  The essence of the argument is to
  provide uniquely the dashed arrow in
  $$\xymatrix{A \ar[rrd]\ar[rd]\ar@{->>}[d] &&\\
  S \ar[rrd]|\hole\ar@{-->}[rd] & \cdot \ar[r]\ar@{->>}[d] & \cdot \ar@{->>}[d] \\
  & X \ar@{ >->}[r] & Y
  }$$
  which amounts to filling
  $$\xymatrix{ A \ar[r]\ar@{->>}[d] & X \ar@{ >->}[d] \\
  S \ar@{-->}[ru] \ar[r] & Y ,}$$
  in turn uniquely fillable by  orthogonality $E \bot F$.
\end{proof}

\begin{lemma}\label{coreflection:w}
  The inclusion $\Ar^E(\DD)\to \Ar(\DD)$ admits a right adjoint $w$.
  This right adjoint $w :\Ar(\DD) \to \Ar^E(\DD)$ sends an arrow $a$ to its $E$-factor.
  In other words, if $a$ factors as $a= f\circ e$ then $w(a)=e$.  
\end{lemma}

\begin{proof}
  This is dual to \cite[5.2.8.19]{Lurie:HTT}.
\end{proof}

\begin{lemma}
  The right adjoint $w$ sends cartesian arrows 
in $\Ar(\DD)$ 
to cartesian arrows
in $\Ar^E(\DD)$.
\end{lemma}

\begin{proof}
  This can be seen from the factorisation:
$$
\vcenter{\xymatrix{
  \cdot \ar[r] \ar[d] & \cdot \ar[d] \\
  \cdot 
  \ar[r]^\sim & \cdot}}\qquad\mapsto\qquad
\vcenter{
\xymatrix{  \cdot \ar@{->>}[d] \ar[r] & \cdot \ar@{->>}[d] \\
  \cdot \ar[r] \ar@{ >->}[d] & \cdot \ar@{ >->}[d] \\
  \cdot 
  \ar[r]^\sim & \cdot}
}
$$
The middle horizontal arrow is forced into $F$ by the closure properties of 
right classes.
\end{proof}

Let $\Fun'(\Lambda^1_2, \DD) = \Ar^E(\DD) \times_{\DD} \Ar^F(\DD)$ denote the 
$\infty$-category whose objects are pairs of composable arrows where the first
arrow is in $E$ and the second in $F$.  Let $\Fun'(\Delta[2],\DD)$ denote the
$\infty$-category of $2$-simplices in $\DD$ for which the two `short' edges are in
$E$ and $F$ respectively.  The projection map 
$\Fun'(\Delta[2],\DD) \to \Fun'(\Lambda^1_2,\DD)$ is always a trivial Kan 
fibration, just because $\DD$ is an $\infty$-category.
\begin{prop}\label{HTT:5.2.8.17}
	{\rm (\cite[5.2.8.17]{Lurie:HTT}.)}
   The projection $\Fun'(\Delta[2],\DD) \to 
  \Fun(\Delta[1],\DD)$ induced by the long edge $d^1 : [1] \to [2]$
  is a trivial Kan fibration.
\end{prop}

\begin{cor}\label{cor:ArD=EF}
  There is an equivalence of $\infty$-categories
  $$
  \Ar(\DD) \isopil \Ar^E(\DD) \times_{\DD} \Ar^F(\DD)
  $$
  given by $(E,F)$-factoring an arrow.
\end{cor}
\begin{proof}
  Pick a section  to the map in \ref{HTT:5.2.8.17} and compose with the 
  projection discussed just prior.
\end{proof}

\bigskip

Let $x$ be an object in $\DD$, and denote by $\DD^E_{x/}$ the $\infty$-category
of $E$-arrows out of $x$.  More formally it is given by the pullback 
$$\xymatrix{
  \DD^E_{x/} \drpullback \ar[r] \ar[d] & \Ar^E(\DD) \ar[d]^{\mathrm{dom}} \\
  {*} \ar[r]_{\name{x}} & \DD}$$
\begin{cor}\label{cor:DxF}
  We have a pullback
  $$\xymatrix{
  \DD_{x/} \drpullback \ar[r] \ar[d] & \Ar^F(\DD) \ar[d]^{\mathrm{dom}} \\
  \DD^E_{x/} \ar[r] & \DD}$$
 \end{cor}
\begin{proof}
  In the diagram
  $$\xymatrix{
    \DD_{x/} \drpullback \ar[r] \ar[d] & \Ar(\DD) \drpullback \ar[r]\ar[d]_w & \Ar^F(\DD) \ar[d]^{\mathrm{dom}} \\
  \DD^E_{x/} \drpullback \ar[r] \ar[d]& \Ar^E(\DD) \ar[r]_{\mathrm{codom}} \ar[d]^{\mathrm{dom}} & \DD\\
  {*} \ar[r]_{\name{x}} & \DD
  }$$
  the right-hand square is a pullback by \ref{cor:ArD=EF};
  the bottom square and the left-hand rectangle are clearly pullbacks, hence the
  top-left square is a pullback, and hence the top rectangle is too.
\end{proof}

\begin{lemma}\label{lem:Dx'Dx}
  Let $e: x\to x'$ be an arrow in the class $E$. Then we have a pullback square
  $$\xymatrix{
  \DD_{x'/} \drpullback \ar[d]_w \ar[r]^{e\uppershriek} & \DD_{x/} \ar[d]_w \\
  \DD^E_{x'/}   \ar[r]_{e\uppershriek} & \DD^E_{x/}
  }$$
  Here $e\uppershriek$ means `precompose with $e$'.
\end{lemma}

\begin{proof}
  In the diagram 
    $$\xymatrix{
  \DD_{x'/} \ar[d]_w \ar[r]^{e\uppershriek} & \DD_{x/} \drpullback \ar[d]_w  
  \ar[r] & \Ar^F(\DD) \ar[d]^{\mathrm{dom}}\\
  \DD^E_{x'/}   \ar[r]_{e\uppershriek} & \DD^E_{x/} \ar[r]_{\mathrm{codom}} & \DD
  }$$
  the functor $\DD_{x/}\to \Ar^F(\DD)$ is `taking $F$-factor'.
  Note that the horizontal composites are again `taking $F$-factor' and 
  codomain, respectively, since 
  precomposing with an $E$-map does not change the $F$-factor.
  Since both the right-hand  square and the rectangle are pullbacks by 
  \ref{cor:DxF},
  the left-hand square is a pullback too.
\end{proof}

\begin{blanko}{Restriction.}\label{setup-end}
  We shall need a slight variation of these results.  We continue the assumption
that $\DD$ is an $\infty$-category with a factorisation system $(E,F)$.
Given a full subcategory $\A \subset \DD$,  we denote by
$\A\comma \DD$ the `comma $\infty$-category of arrows in $\DD$ with domain in $\A$'.  More precisely
it is defined as the pullback 
$$\xymatrix{
\A\comma \DD \drpullback \ar[d]_{\mathrm{dom}}\ar[r]^{\mathrm{f.f}} & \Ar(\DD) \ar[d]^{\mathrm{dom}} \\
\A \ar[r]_{\mathrm{f.f}} & \DD}$$
The map $\A\comma \DD \to \A$ is a cartesian fibration.
Similarly, let $\Ar^E(\DD)_{|\A}$ denote the comma $\infty$-category of $E$-arrows with domain 
in $\A$, defined as the pullback 
$$\xymatrix{
\Ar^E(\DD)_{|\A} \drpullback \ar[d]_{\mathrm{dom}}\ar[r]^{\mathrm{f.f}} & \Ar^E(\DD) \ar[d]^{\mathrm{dom}} \\
\A \ar[r]_{\mathrm{f.f}} & \DD}$$
Again $\Ar^E(\DD)_{|\A} \to \A$ is a cartesian fibration (where the cartesian arrows
are squares whose top part is in $\A$ and whose bottom horizontal arrow belongs
to the class $E$).  These two fibrations are just the restriction to $\A$ of
the fibrations $\Ar(\DD) \to \DD$ and $\Ar^E(\DD) \to \DD$. Since the coreflection
$\Ar(\DD) \to \Ar^E(\DD)$ is vertical for the domain fibrations, it restricts to
a coreflection $w:\A\comma \DD \to \Ar^E(\DD)_{|\A}$.

Just as in the unrestricted  situation (Corollary~\ref{cor:ArD=EF}),
we have a pullback square
$$
\xymatrix{\A\comma \DD \drpullback \ar[r] \ar[d]_w & \Ar^F(\DD) \ar[d] \\
\Ar^E(\DD)_{|\A} \ar[r] & \DD}
$$
saying that an arrow in $\DD$ factors like before, also if it starts in
an object in $\A$.
Corollary~\ref{cor:DxF} is the same in the restricted situation --- just assume
that $x$ is an object in $\A$.   Lemma~\ref{lem:Dx'Dx} is also the same, just
assume that $e:x'\to x$ is an $E$-arrow between $\A$-objects.
\end{blanko}

\bigskip

The following easy lemma expresses the general idea of extending a 
factorisation system.

\begin{lemma}\label{fact-ext}
  Given an adjunction $\xymatrix{L:\DD \ar@<+3pt>[r] & \ar@<+3pt>[l] \CC:R}$
  and given
  a factorisation system $(E,F)$ on $\DD$ with
  the properties

  --- $RL$ preserves the class $F$;

  --- $R\varepsilon$ belongs to $F$;

  \noindent
  consider the full subcategory $\wtil \DD \subset \CC$ spanned by the image of
  $L$.  
  Then there is an induced factorisation system $(\wtil E, \wtil F)$ on $\wtil
  \DD\subset \CC$ with $\wtil E := L(E)$ (saturated by equivalences), and $\wtil
  F := R^{-1}F \cap \wtil \DD$.
\end{lemma}

\begin{proof}
  It is clear that the classes $\wtil E$ and $\wtil F$ are closed under
  equivalences.  The two classes are orthogonal: given $Le\in \wtil E$ and
  $\tilde f \in \wtil F$ we have $Le \bot \tilde f$ in the full subcategory
  $\wtil\DD \subset\CC$ if and only if $e \bot R\tilde f$ in $\DD$, and the
  latter is true since $R\tilde f \in F$ by definition of $\wtil F$.  Finally,
  every map $g: LA \to X$ in $\wtil\DD$ admits an $(\wtil E, \wtil
  F)$-factorisation: indeed, it is transpose to a map $A \to RX$, which we
  simply $(E,F)$-factor in $\DD$,
  $$
  \xymatrix@R-1em@C+1em{A\ar[rr]\ar[rd]_e&&RX,\\&D\ar[ru]_f}
  $$
  and transpose back the factorisation (i.e.~apply $L$ and postcompose with the counit):
  $g$ is now the composite
  $$
  \xymatrix{ LA \ar[r]^{Le} & LD \ar[r]^{Lf} & LRX \ar[r]^\varepsilon & X ,}
  $$
  where clearly $Le \in \wtil E$, and we also have $\varepsilon \circ Lf \in \wtil F$
  because of the two conditions imposed.
\end{proof}

\begin{blanko}{Remarks.}\label{fact-decomp}
  By general theory (\ref{coreflection:w}),
  having the factorisation system $(\wtil E , \wtil F)$ 
  implies the existence of a right adjoint to the inclusion
  $$
  \Ar^{\wtil E}(\wtil\DD)  \longrightarrow \Ar(\wtil \DD).
  $$
  This right adjoint returns the $\wtil E$-factor of an arrow.

  Inspection of the proof of \ref{fact-ext} shows that we have the 
  same factorisation property 
  for other maps in $\CC$ than those between objects in $\Im L$, namely
  giving up the requirement that the codomain
  should belong to $\Im L$: it is enough that the domain belongs to $\Im L$:
  {\em every map in $\CC$ whose domain belongs to $\Im L$
  factors as a map in $\wtil E$ followed by a map in $\wtil F:=R^{-1}F$, and we
  still have $\wtil E \bot \wtil F$, without restriction on the codomain
  in the right-hand class.}  This result amounts to a coreflection:
\end{blanko}

\begin{theorem}\label{fact-theorem}
  In the situation of Lemma~\ref{fact-ext}, let $\wtil\DD \comma\CC \subset \Ar(\CC)$
  denote the {\em full} subcategory spanned by the maps with domain in $\Im L$.
  The inclusion functor
$$
\Ar^{\wtil E}(\wtil\DD) \into \wtil\DD\comma\CC
$$
has a right adjoint,
given by factoring any map with domain in $\Im L$
and returning the $\wtil E$-factor.
Furthermore, the right adjoint preserves cartesian arrows (for the domain 
projections).
\end{theorem}

\begin{proof}
  Given that the factorisations exist as explained above, the proof now follows
  the proof of Lemma~5.2.8.18 in Lurie~\cite{Lurie:HTT}, using the dual of his
  Proposition~5.2.7.8.
\end{proof}

\bigskip

The following restricted version of these results will be useful.

\begin{lemma}\label{fact-Kl-E}
  In the situation of Lemma~\ref{fact-ext}, assume
   there is a full subcategory $J:\A \into\DD$ such that
  
  --- All arrows in $\A$ belong to $E$.
  
  --- If an arrow in $\DD$ has its domain in $\A$, then its $E$-factor
  also belongs to $\A$.
    
  \noindent
  Consider the full subcategory $\wtil \A \subset \CC$ spanned by the image of
  $LJ$. 
  Then there is induced a factorisation system $(\wtil E, \wtil F)$ on $\wtil
  \A\subset \CC$ with $\wtil E := LJ(E)$ (saturated by equivalences) and $\wtil
  F := R^{-1}F \cap \wtil \A$.
\end{lemma}

\begin{proof}
    The proof is the same as before.
\end{proof}

\medskip

\begin{blanko}{A basic factorisation system.}\label{eq1-cart}
    Suppose $\CC$ is any $\infty$-category with pullbacks, and $\DD$ is an $\infty$-category 
    with a terminal object $\terminal$.  Then evaluation on
    $\terminal$ defines a cartesian fibration
    $$\operatorname{ev}_\terminal :\Fun(\DD,\CC) \to \CC$$ 
    for which the cartesian arrows are precisely the cartesian natural
    transformations.  The vertical arrows are the natural 
    transformations whose component at $\terminal$ is an equivalence.
    Hence the functor $\infty$-category has a factorisation system in which the
    left-hand class is the class of vertical natural transformations, 
    and the right-hand class is the class of cartesian natural 
    transformations:
  $$
  \xymatrix@R-1em@C+1em{X\ar[rr]\ar[rd]_{\mathrm{eq.\, on\; 1}}&&Y\\
  &Y'\ar[ru]_{\mathrm{cartesian}}}
  $$
\end{blanko}

\bigskip

Finally we shall need the following general result (not related to factorisation 
systems):
\begin{lemma}\label{preservespullback}
  Let $\DD$ be any $\infty$-category.  Then the functor
  \begin{eqnarray*}
    F: \DD\op & \longrightarrow & \Grpd  \\
    D & \longmapsto & (\DD_{D/})^{\eq},
  \end{eqnarray*}
  corresponding to the right fibration $\Ar(\DD)^{\mathrm{cart}}\to\DD$,
  preserves pullbacks.
\end{lemma}

\begin{proof}
  Observe first that $F \simeq \colim_{X\in \DD^\eq} \Map ( - , X)$,
  a homotopy sum of representables.  
  Given now a pushout in $\DD$,
  $$\xymatrix{ D \drpullback & \ar[l] B \\
\ar[u] A & \ar[l] \ar[u] C}$$
  we compute, using the distributive law:
  \begin{eqnarray*}
     F(A \coprod_C B )& \simeq & \colim_{X\in \DD^\eq} \Map ( A \coprod_C B , X)   \\
 & \simeq & \colim_{X\in \DD^\eq} \big( \Map (A,X) \times_{\Map(C,X)} \Map(B,X) \big)  \\
 & \simeq & \colim_{X\in \DD^\eq}  \Map (A,X) \times_{\colim\Map(C,X)} \colim_{X\in \DD^\eq}\Map(B,X)  \\
 & \simeq & F(A) \times_{F(C)} F(B) .
  \end{eqnarray*}
\end{proof}

\section{Flanked decomposition spaces}
\label{sec:flanked}

\begin{blanko}{Idea.}
  The idea is that `interval' should mean complete decomposition space (equipped)
  with both an initial and a terminal object.  An object $x\in X_0$ is {\em
  initial} if the projection map $X_{x/} \to X$ is a levelwise equivalence.
  Here the {\em coslice} $X_{x/}$ is defined as the pullback of the lower dec
  $\Decbot{X}$ along $\terminal \stackrel{\name{x}} \to X_0$.  Terminal objects are
  defined similarly with slices, i.e.~pullbacks of the upper dec.  It is not
  difficult to see (compare Proposition~\ref{prop:i*flanked=Segal} below) that
  the existence of an initial or a terminal object forces $X$ to be a Segal
  space.
  
  While this intuition may be helpful,
  it turns out to be practical
  to approach the notion of interval from a more abstract viewpoint, which will
  allow us to get hold of various adjunctions and factorisation systems that are
  useful to prove things about intervals.  We come to intervals in the next 
  section.  First we have to deal with flanked decomposition spaces.
\end{blanko}

\begin{blanko}{The category $\Xi$ of finite strict intervals.}
  We denote by $\Xi$ the category of finite strict intervals
  (cf.~\cite{Joyal:disks}), that is, a skeleton of the category whose objects
  are nonempty finite linear orders with a bottom and a top element, required to
  be distinct, and whose arrows are the maps that preserve both the order and
  the bottom and top elements.  We imagine the objects as columns of dots, with
  the bottom and top dot white, then the maps are the order-preserving maps that
  send white dots to white dots, but are allowed to send black dots to white
  dots.

  There is a forgetful functor $u:\Xi\to\simplexcategory$ which forgets that there is anything
  special about the white dots, and just makes them black.  This functor has a left
  adjoint $i:\simplexcategory\to\Xi$ which to a linear order (column of black dots) adjoins a
  bottom and a top element (white dots).

  Our indexing convention for $\Xi$ follows
  the free functor $i$: the object in $\Xi$ with $k$ black dots (and two outer white
  dots) is denoted $[k-1]$.  Hence the objects in $\Xi$ are $[-1]$, $[0]$, $[1]$, etc.
  Note that $[-1]$ is an initial object in $\Xi$.
  The two functors can therefore be described on objects as $u([k])=[k+2]$
  and $i([k])=[k]$, and the adjunction is given by the following 
  isomorphism:
  \begin{equation}\label{eq:XiDelta}
    \Xi([n],[k]) = \simplexcategory([n],[k\!+\!2]) \qquad n\geq 0, k\geq -1 .
  \end{equation}
\end{blanko}

\begin{blanko}{New outer degeneracy maps.}
  Compared to $\simplexcategory$ via the inclusion $i:\simplexcategory\to \Xi$, the category $\Xi$ 
  has one extra 
  coface map in $\Xi$,
  namely $[-1]\to [0]$.  It also has, in each degree, 
  two extra {\em outer codegeneracy 
  maps}:  $s^{\bot-1}: [n]\to [n-1]$ sends the bottom black dot to the bottom white 
  dot, and $s^{\top+1} : [n]\to [n-1]$ sends the top black dot to the top white 
  dot.  (Both maps are otherwise bijective.)  
\end{blanko}

\begin{blanko}{Basic adjunction.}
  The adjunction $i \isleftadjointto u$ induces an adjunction $i\upperstar
  \isleftadjointto u\upperstar$
$$\xymatrix{
\Fun(\Xi\op,\Grpd) \ar@<+3pt>[r]^-{i\upperstar} & \ar@<+3pt>[l]^-{u\upperstar} 
\Fun(\simplexcategory\op,\Grpd)
}$$
which will play a central role in all the constructions in this section.

The functor $i\upperstar$ takes {\em 
underlying simplicial space}:
concretely, applied to a $\Xi\op$-space $A$, the functor $i\upperstar $ deletes $A_{-1}$
and removes all the extra outer degeneracy maps.

On the other hand, the functor $u\upperstar $,
applied to a simplicial space $X$, 
deletes $X_0$ and removes all outer face maps (and then reindexes).

The comonad 
$$
i\upperstar u\upperstar : \Fun(\simplexcategory\op,\Grpd) \to \Fun(\simplexcategory\op,\Grpd)
$$
is precisely the double-dec construction $\Decbot{\Dectop{}}$, and the
counit of the adjunction is precisely the
  double-dec map
$$
\varepsilon_X=d_\top d_\bot:i\upperstar u\upperstar X=\Decbot{\Dectop{X}}
\longrightarrow X .
$$

On the other hand, the monad
$$
u\upperstar i\upperstar : \Fun(\Xi\op,\Grpd) \to \Fun(\Xi\op,\Grpd)
$$
is also a kind of double-dec, removing first the extra outer degeneracy
maps, and then the outer face maps.  The unit 
$$\eta_A = s_{\bot-1} s_{\top+1}: A \to 
u\upperstar i\upperstar A$$ 
will also play an important role.
\end{blanko}

\begin{lemma}\label{u*f(cULF)=cart}
    If $f: Y \to X$ is a \culf map of simplicial spaces,
  then  $u\upperstar f :u\upperstar Y \to u\upperstar X $ is cartesian.
\end{lemma}
\begin{proof}
  The \culf condition on $f$ says
  it is cartesian on `everything' except outer face 
  maps, which are thrown away when taking $u\upperstar f$.
\end{proof}
\noindent
Note that the converse is not always true: if $u\upperstar f$ is cartesian then
$f$ is ULF, but there is no information about $s_0: Y_0 \to Y_1$, so we 
cannot conclude that $f$ is conservative.

Dually:
\begin{lemma}\label{i*cart}
  If a map of  $\Xi\op$-spaces  $g : B \to A$ is cartesian
 (or just cartesian on inner face and degeneracy maps),
  then $i\upperstar g : i\upperstar B \to i\upperstar A$ is cartesian.
\end{lemma}

\begin{blanko}{Representables.}
  The representables on $\Xi$ we denote by $\Xi[-1]$, $\Xi[0]$, etc.  By
  convention we will also denote the terminal presheaf on $\Xi$ by $\Xi[-2]$,
  although it is not representable since we have chosen not to include $[-2]$ (a
  single white dot) in our definition of $\Xi$.  Note that \eqref{eq:XiDelta}
  says that $i\upperstar$ preserves representables:
  \begin{equation}\label{eq:Deltak+2}
  i\upperstar (\Xi[k]) \simeq \Delta[k\!+\!2], \qquad k\geq-1 .
  \end{equation}
\end{blanko}

\begin{blanko}{Stretched-cartesian factorisation system.}
  Call an arrow in $\Fun(\Xi\op,\Grpd)$ {\em stretched} if its $[-1]$-component is an
  equivalence.  Call an arrow {\em cartesian} if it is a cartesian natural
  transformation of $\Xi\op$-spaces.  By general theory (\ref{eq1-cart}) we
  have a factorisation system on $\Fun(\Xi\op,\Grpd)$ where the left-hand class
  is formed by the stretched maps and the right-hand class consists of the cartesian maps.
  In concrete terms, given any map $B \to
  A$, since $[-1]$ is terminal in $\Xi\op$, one can pull back the whole
  diagram $A$ along the map $B_{-1} \to A_{-1}$.  The resulting
  $\Xi\op$-space $A'$ is cartesian over $A$ by construction,
  and by the universal property of the pullback it receives a map from
  $B$ which is manifestly the identity in degree $-1$, hence
  stretched.
  $$
  \xymatrix@R-1em@C+1em{B\ar[rr]\ar[rd]_{\mathrm{stretched}}&&A\\&A'\ar[ru]_{\mathrm{cartesian}}}
  $$
\end{blanko}

\begin{blanko}{Flanked $\Xi\op$-spaces.}\label{flanked}
  A $\Xi\op$-space $A$ is called \emph{flanked} if the extra outer degeneracy
  maps form cartesian squares with opposite outer face maps. 
  Precisely, for $n\geq 0$
  $$
  \xymatrix{
  A_{n-1}   \ar[d]_{s_{\bot-1}} & \ar[l]_{d_\top} \dlpullback A_n \ar[d]^{s_{\bot-1}} \\
  A_n & \ar[l]^{d_\top} A_{n+1}
  }
  \qquad
  \xymatrix{
  A_{n-1}   \ar[d]_{s_{\top+1}} & \ar[l]_{d_\bot} \dlpullback A_n 
  \ar[d]^{s_{\top+1}} \\
  A_n & \ar[l]^{d_\bot} A_{n+1}
  }
  $$
  Here we have included the special extra face map $A_{-1} \leftarrow A_0$ both as
  a top face map and a bottom face map.
\end{blanko}

\begin{blanko}{Example.}
  If $\CC$ is a small category with an initial object and a terminal
  object, then its nerve is naturally a $\Xi\op$-space (contractible in
  degree $-1$): the extra bottom degeneracy maps add the initial object to
  the beginning of a sequence of composable arrows, and the extra top
  degeneracy maps add the terminal object to the end of a sequence of
  composable arrow.  The flanking condition then states precisely that
  these two objects are initial and terminal.
\end{blanko}

\begin{lemma}\label{s-1-newbonuspullbacks}
  (`Bonus pullbacks' for flanked spaces.)
  In a flanked $\Xi\op$-space $A$, all the following squares are pullbacks:
$$
\xymatrix @C=12pt @C=12pt {
A_{n-1} \ar[d]_{s_{\bot-1}} & \ar[l]_-{d_i} \dlpullback A_n \ar[d]^{s_{\bot-1}} \\
A_n  & \ar[l]^-{d_{i+1}} A_{n+1}
}
\
\xymatrix @C=12pt @C=12pt {
A_{n-1} \drpullback \ar[r]^-{s_j} \ar[d]_{s_{\bot-1}} & A_n \ar[d]^{s_{\bot-1}} \\
A_n \ar[r]_-{s_{j+1}} & A_{n+1}
}
\
\xymatrix @C=12pt @C=12pt {
A_{n-1} \ar[d]_{s_{\top+1}} & \ar[l]_-{d_i} \dlpullback A_n \ar[d]^{s_{\top+1}} \\
A_n  & \ar[l]^-{d_i} A_{n+1}
}
\
\xymatrix @C=12pt @C=12pt {
A_{n-1} \drpullback \ar[r]^-{s_j} \ar[d]_{s_{\top+1}} & A_n \ar[d]^{s_{\top+1}} \\
A_n \ar[r]_-{s_j} & A_{n+1}
}
$$
This is for all $n\geq 0$, and the running indices are $0 \leq i \leq n$ and $-1 \leq j \leq n$.
\end{lemma}

\begin{proof}
  Easy argument with pullbacks, similar to \cite[3.10]{GKT:DSIAMI-1}.
\end{proof}

Note that in the upper rows, all face or degeneracy maps are present, whereas
in the lower rows, there is one map missing in each case.  In particular,
all the `new' outer degeneracy maps appear as pullbacks of `old' degeneracy
maps.

\begin{blanko}{Flanked decomposition spaces.}
  By definition, a {\em flanked decomposition space} is a $\Xi\op$-space
  $A:\Xi\op\to\Grpd$ that is flanked and whose underlying $\simplexcategory\op$-space
  $i\upperstar A$ is a decomposition space.  Let $\FD$ denote the full
  subcategory of $\Fun(\Xi\op,\Grpd)$ spanned by the flanked
  decomposition spaces.
\end{blanko}

\begin{lemma}\label{lem:flanked}
  If $X$ is a decomposition space, then 
  $u\upperstar X$ is a flanked decomposition space.
\end{lemma}
\begin{proof}
  The underlying simplicial space is clearly a decomposition space (in fact a 
  Segal space), since all we have done is to throw away some outer face maps 
  and reindex.  The flanking condition comes from the `bonus pullbacks' of $X$,
  cf.~\cite[3.10]{GKT:DSIAMI-1}.
\end{proof}

It follows that the basic adjunction $i\upperstar \isleftadjointto u\upperstar $
restricts to an adjunction 
$$\xymatrix{
i\upperstar  : \FD \ar@<+3pt>[r] & \ar@<+3pt>[l]\Decomp : u\upperstar 
}$$
between flanked decomposition spaces (certain $\Xi\op$-spaces)
and decomposition spaces.

\begin{lemma}\label{counit-cULF}
  The counit
  $$
  \varepsilon_X : i\upperstar u\upperstar X \to X
  $$
  is \culf, when $X$ is a decomposition space.
\end{lemma}
\begin{proof}
  This follows from Theorem~\ref{thm:decomp-dec-segal}.
\end{proof}

\begin{lemma}\label{unit-cart}
  The unit
  $$
  \eta_A: A \to u\upperstar i\upperstar A
  $$
  is cartesian, when $A$ is flanked.
\end{lemma}
\begin{proof}
  The map $\eta_A$ is given by $s_{\bot-1}$ followed by $s_{\top+1}$.
  The asserted pullbacks are precisely the `bonus pullbacks' of 
  Lemma~\ref{s-1-newbonuspullbacks}.
\end{proof}

From Lemma~\ref{unit-cart} and Lemma~\ref{counit-cULF} we get:

\begin{cor}\label{u*i*(cart)=cart}
  The monad $u\upperstar i\upperstar : \FD \to \FD$ preserves cartesian maps.
\end{cor}

\begin{lemma}\label{cart&cULF}
  $i\upperstar A \to X$ is \culf in $\Decomp$ if and only if the transpose
  $A \to u\upperstar  X$ is cartesian in $\FD$.
\end{lemma}
\begin{proof}
  This follows since the unit is cartesian (\ref{unit-cart}), the counit is 
  \culf (\ref{counit-cULF}), and $u\upperstar $ and $i\upperstar $ send those
  two classes to each other (\ref{u*f(cULF)=cart} and \ref{i*cart}).
\end{proof}

\begin{prop}\label{prop:i*flanked=Segal}
  If $A$ is a flanked decomposition space, then $i\upperstar A $ is a Segal 
  space.
\end{prop}
\begin{proof}
  Put $X = i\upperstar A$.  We have the maps
  $$
  \xymatrix{
  i\upperstar A \ar[r]^-{i\upperstar  \eta_{A}} 
  & i\upperstar u\upperstar i\upperstar  A =
  u\upperstar  i\upperstar X \ar[r]^-{\varepsilon_{X}} 
  & X= i\upperstar 
  A}
  $$
  Now $X$ is a decomposition space by assumption, so $i\upperstar 
  u\upperstar X = \Decbot{\Dectop{X}}$ is a Segal space and the counit is \culf
  (both statements by Theorem~\ref{thm:decomp-dec-segal}).
  On the other hand, since $A$ is flanked, the unit $\eta$ is 
  cartesian by Lemma~\ref{unit-cart}, hence $i\upperstar \eta$ is 
  cartesian by Lemma~\ref{i*cart}.  Since $i\upperstar A$ is thus cartesian over a Segal 
  space, it is itself a Segal space (\cite[2.11]{GKT:DSIAMI-1}).
\end{proof}

\begin{lemma}\label{Flanked/cart} 
  If $B\to A$ is a cartesian map of $\Xi\op$-spaces and $A$ is a flanked
  decomposition space then so is $B$.
\end{lemma}

\begin{cor}
  The stretched-cartesian factorisation system
  restricts to a factorisation system on $\FD$.
\end{cor}

\begin{lemma}\label{repr=flanked}
  The representable functors $\Xi[k]$ are flanked.
\end{lemma}

\begin{proof}
  Since the pullback squares required for a presheaf to be flanked
  are images of pushouts in $\Xi$, this follows since representable
  functors send colimits to limits.
\end{proof}

\section{Intervals and the factorisation-interval construction}
\label{sec:fact-int}

\begin{blanko}{Complete $\Xi\op$-spaces.}
  A $\Xi\op$-space is called {\em complete} if all degeneracy maps are
  monomorphisms.  We are mostly interested in this notion for flanked
  decomposition spaces.  In this case, if just $s_0: A_0 \to A_1$ is a
  monomorphism, then all the degeneracy maps are monomorphisms.  This follows
  because on the underlying decomposition space, we know 
  \cite[2.5]{GKT:DSIAMI-2} that
  $s_0: A_0 \to A_1$ being a monomorphism implies that all the simplicial
  degeneracy maps are monomorphisms, and by flanking we then deduce that also the new
  outer degeneracy maps are monomorphisms.  Denote by $\cFD \subset\FD$ the full
  subcategory spanned by the complete flanked decomposition spaces.
  
  It is clear that if $X$ is a complete decomposition space, then $u\upperstar X$
  is a complete flanked decomposition space, and if $A$ is a complete flanked
  decomposition space then $i\upperstar A $ is a complete decomposition space.
  Hence the fundamental adjunction
  $\xymatrix{
  i\upperstar  : \FD \ar@<+3pt>[r] & \ar@<+3pt>[l]\Decomp : u\upperstar 
  }$
  between flanked decomposition
  spaces and decomposition spaces restricts to an adjunction
  $$\xymatrix{
  i\upperstar  : \cFD \ar@<+3pt>[r] & \ar@<+3pt>[l]\cDecomp : u\upperstar 
  }$$
 between complete flanked decomposition
  spaces and complete decomposition spaces.
  
  Note that anything cartesian over a 
  complete $\Xi\op$-space is again complete.
\end{blanko}

\begin{blanko}{Reduced $\Xi\op$-spaces.}
  A $\Xi\op$-space $A:\Xi\op\to\Grpd$ is called {\em reduced} when $A[-1]\simeq 
  \terminal$.
\end{blanko}

\begin{lemma}\label{reduced/wide}
  If $A\to B$ is a stretched map of $\Xi\op$-spaces and $A$ is reduced then $B$ is 
  reduced.
\end{lemma}

\begin{blanko}{Algebraic intervals.}\label{aInt}
  An {\em algebraic interval} is  by definition a reduced complete flanked 
  decomposition space.
  We denote by $\aInt$ the full subcategory of $\Fun(\Xi\op,\Grpd)$
  spanned by the algebraic intervals.  In other words, a {\em morphism of
  algebraic intervals} is just a natural transformation of functors
  $\Xi\op\to\Grpd$.
  Note that the underlying decomposition space of an interval is always a Segal
  space.
\end{blanko}

\begin{lemma}
  All representables $\Xi[k]$ are algebraic intervals (for $k\geq -1$), 
  and also the terminal presheaf $\Xi[-2]$ is an algebraic interval.
\end{lemma}

\begin{proof}
  It is clear that all these presheaves are contractible in degree $-1$, and
  they are flanked by Lemma~\ref{repr=flanked}.  It is also clear from
  \eqref{eq:Deltak+2} that their underlying simplicial spaces are complete
  decomposition spaces (they are even Rezk complete Segal spaces).
\end{proof}

\begin{lemma}
  $\Xi[-1]$ is an initial object in $\aInt$.
\end{lemma}

\begin{lemma}\label{aInt:map=wide}
  Every morphism in $\aInt$ is stretched.
\end{lemma}

\begin{cor}
  If a morphism of algebraic intervals is cartesian, then it is an equivalence.
\end{cor}

\begin{blanko}{The factorisation-interval construction.}
  We now come to the important notion of factorisation
  interval $I(a)$ of a given arrow $a$ in a decomposition space $X$. 
  In the case where $X$ is a $1$-category the construction is due to
  Lawvere~\cite{Lawvere:statecats}:
  the objects of $I(a)$ are the two-step
  factorisations of $a$, with initial object id-followed-by-$a$ and terminal object
  $a$-followed-by-id.  The $1$-cells are arrows between such
  factorisations, or equivalently $3$-step factorisations, and so on.

  For a general (complete) decomposition space $X$, the idea is this:
  taking the double-dec of $X$ gives a simplicial object starting at $X_2$, but
  equipped with an augmentation $X_1 \leftarrow X_2$.   Pulling back this 
  simplicial object along $\name a: \terminal \to X_1$ yields a new simplicial object
  which is $I(a)$.  This idea can be formalised in terms of the basic adjunction
  as follows.
  
  By Yoneda, to give an arrow $a\in X_1$ is to give $\Delta[1] \to X$ in
  $\Fun(\simplexcategory\op,\Grpd)$, or in the full subcategory $\cDecomp$.
  By adjunction, this is equivalent to giving $\Xi[-1] \to u\upperstar X$ in
  $\cFD$.  Now factor this map as a stretched map followed by a cartesian map:
  $$\xymatrix{
  \Xi[-1] \ar[rr] \ar[rd]_{\mathrm{stretched}} && u\upperstar X  .\\
  & A \ar[ru]_{\mathrm{cart}} &
  }$$
  The object appearing in the middle is an algebraic
  interval since it is stretched under $\Xi[-1]$ (\ref{reduced/wide}).
  By definition, the factorisation interval of $a$ is
  $I(a) := i\upperstar A$, equipped with a \culf map to $X$, as seen in the diagram
  $$\xymatrix{
  \Delta[1] \ar[rr] \ar[rd] && i\upperstar u\upperstar X 
  \ar[r]^-{\varepsilon}_-{\culfsymb} & X .\\
  & I(a) \ar[ru]_{\culfsymb} &
  }$$
  The map $\Delta[1] \to I(a)$ equips $I(a)$ with two endpoints, and a longest 
  arrow between them.
  The \culf map $I(a) \to X$ sends the longest arrow of $I(a)$ to $a$.
  
  More generally, by the same adjunction argument,
  given an $k$-simplex $\sigma: \Delta[k]\to X$ with long edge $a$,
  we get a $k$-subdivision of $I(a)$, i.e.~a stretched map $\Delta[k]\to I(a)$.
  
  The construction shows, remarkably, that 
  as far as comultiplication is concerned, any decomposition space is locally
  a Segal space, in the sense that the comultiplication of an arrow $a$
  may as well be performed inside $I(a)$, which is a Segal space by 
  \ref{prop:i*flanked=Segal}.
  So while there may be no global way to compose arrows even if their source
  and targets match, the {\em decompositions} that exist do compose again.
\end{blanko}

We proceed to formalise the factorisation-interval construction. 

\begin{blanko}{Coreflections.}
  Inside the
  $\infty$-category of arrows $\Ar(\cFD)$, denote by
  $\Ar^{\sker}(\cFD)$ the full subcategory spanned by the
  stretched maps.  The stretched-cartesian factorisation system amounts to a coreflection
  $$
  w: 
  \Ar(\cFD) \longrightarrow \Ar^{\sker}(\cFD) ;
  $$
  it sends an arrow $A\to B$ to its stretched factor $A \to B'$, and in particular can
  be chosen to have $A$ as domain again (\ref{coreflection:w}).
  In particular, for each algebraic interval $A\in\aInt\subset\cFD$,
  the adjunction restricts to an adjunction between coslices, with coreflection
  $$
  w_A: \cFD _{A/} \longrightarrow \cFD^{\sker}_{A/}  .
  $$
  The first $\infty$-category is that of flanked decomposition spaces under $A$, and the
  second $\infty$-category is that of flanked decomposition spaces with a stretched map from
  $A$.  Now, if a flanked decomposition space receives a
  stretched map from an algebraic interval then it is itself an algebraic interval
  (\ref{reduced/wide}), and all maps of algebraic intervals are stretched
  (\ref{aInt:map=wide}).  So in the end the cosliced adjunction takes the form
  of the natural full inclusion functor
  $$
  v_A: \aInt_{A/} \to \cFD_{A/}
  $$
  and a right adjoint
  $$
  w_A: \cFD_{A/} \to \aInt_{A/} .
  $$
  
\begin{blanko}{Remark.}
  These observations amount to saying that
  the functor $v: \aInt \to \cFD$ is a {\em colocal left adjoint}.
  This notion is dual to the important concept
  of local right adjoint~\cite{Weber:TAC18,Gepner-Haugseng-Kock}.
\end{blanko}
  
We record the following obvious lemmas:
\begin{lemma}
  The coreflection $w$ sends cartesian maps to equivalences.
\end{lemma}
\begin{lemma}
  The counit is cartesian.
\end{lemma}

\begin{blanko}{Factorisation-interval as a comonad.}\label{I=LvwR}
  We also have the basic adjunction $i\upperstar \isleftadjointto u\upperstar $
  between complete decomposition spaces and complete flanked decomposition spaces.
  Applied to coslices over an algebraic interval $A$ and its underlying 
  decomposition space $\un A =i\upperstar A$, we get the 
  adjunction
  $$\xymatrix{
  L : \cFD_{A/} \ar@<+3pt>[r] & \ar@<+3pt>[l]\cDecomp_{\un A/} : R   .
  }$$
  Here $L$ is simply the functor $i\upperstar$,
  while the right adjoint $R$ is given by applying $u\upperstar$
  and precomposing with the unit $\eta_A$.
  Note that the unit of this adjunction $L \isleftadjointto R$ at an object
  $f:A \to X$ is given by
  $$
  \xymatrix {& A \ar[ld]_f \ar[rd]^{u\upperstar i\upperstar f \circ \eta_A} \\
  X \ar[rr]_{\eta_X} && u\upperstar i\upperstar  X
  }
  $$
  
  We now combine the two adjunctions:
  $$\xymatrix{
  \aInt_{A/} \ar@<+3pt>[r]^v & \ar@<+3pt>[l]^w \cFD_{A/} 
  \ar@<+3pt>[r]^-L
  & \ar@<+3pt>[l]^-R \cDecomp_{\un A/}  \, .
  }$$

  The factorisation-interval functor is the $\un A= \Delta[k]$ instantiation:
  $$
  I := L \circ v \circ w \circ R  .
  $$
  Indeed, this is precisely what we said in the construction, 
  just phrased more functorially.
  It follows that the factorisation-interval construction is a comonad
  on $\cDecomp_{\un A/}$.
\end{blanko}

\begin{lemma}
  The composed counit is \culf.
\end{lemma}
\begin{proof}
  This follows readily from \ref{counit-cULF}.
\end{proof}

\end{blanko}

\begin{prop}
  The composed unit $\eta : \Id \Rightarrow w \circ R \circ L \circ v$ is an
  equivalence.
\end{prop}

\begin{proof}
  The result of applying the four functors to an algebraic interval map $f:A\to B$
  is the stretched factor in
    $$\xymatrix{
  A \ar[rr] \ar[rd]_{\mathrm{stretched}} && u\upperstar i\upperstar B \\
  & D \ar[ru]_{\mathrm{cart}} &
  }$$
  The unit on $f$ sits in this diagram
  $$
  \xymatrix{
  & \ar[ld]_f  A  \ar[rd] & \\
  B \ar@{-->}[rr]_{\eta_f} \ar[rd]_{\eta_B}&& D \ar[ld]\\
  & u\upperstar i\upperstar B ,}
  $$
  where $\eta_B$ is cartesian by \ref{unit-cart}.
  It follows now from 
  orthogonality of the
  stretched-cartesian factorisation system
  that $\eta_f$ is an equivalence.
\end{proof}

\begin{cor}\label{i*v:ff}
  The functor $i\upperstar \circ v : \aInt \to \cDecomp_{\Delta[1]/}$
  is fully faithful.
\end{cor}

\begin{prop}\label{prop:I(cULF)=eq}
  $I$ sends \culf maps to equivalences.  In detail, for a \culf map $F: Y\to 
  X$ and any arrow $a \in Y_1$
  we have a natural equivalence of intervals (and hence of underlying Segal spaces)
  $$
  I(a) \isopil I(Fa).
  $$
\end{prop}
\noindent
\begin{proof}
  $R$ sends \culf maps to cartesian maps, and $w$ send cartesian maps to 
  equivalences.
\end{proof}

\begin{cor}
  If $X$ is an interval, with longest arrow $a\in X_1$, then $X\simeq I(a)$.
\end{cor}

\begin{prop}\label{aInt-faithful}
  The composed functor 
  $$
  \aInt \to \cDecomp_{\Delta[1]/} \to \cDecomp
  $$
  is faithful (i.e.~induces a monomorphism on mapping spaces).
\end{prop}
\begin{proof}
  Given two algebraic intervals $A$ and $B$, 
  denote by $f:\Delta[1] \to i\upperstar A$ and $g: \Delta[1] \to i\upperstar B$
  the images in $\cDecomp_{\Delta[1]/}$.
  The claim is that the map
  $$
  \Map_{\aInt}(A,B) \longrightarrow \Map_{\cDecomp_{\Delta[1]/}}(f,g) \longrightarrow
  \Map_{\cDecomp}(i\upperstar A, i\upperstar B)
  $$
  is a monomorphism.
  We already know that the first part is an equivalence (by Corollary~\ref{i*v:ff}).
  The second map will be a monomorphism because of the special nature of $f$ and $g$.
  We have a pullback diagram (mapping space fibre sequence for coslices):
  $$\xymatrix{
     \Map_{\cDecomp_{\Delta[1]/}}(f,g)\drpullback \ar[r]\ar[d] & 
     \Map_{\cDecomp}(i\upperstar A, i\upperstar B)\ar[d]^{\mathrm{precomp. }f} \\
     \terminal \ar[r]_-{\name g} & \Map_{\cDecomp}(\Delta[1], i\upperstar B).
  }$$
  Since $g:\Delta[1]\to i\upperstar B$ is the image of the canonical map 
  $\Xi[-1]\to B$,
  the map 
  $$
  \xymatrix{
  \terminal \ar[r]^-{\name g}& \Map_{\cDecomp}(\Delta[1], i\upperstar B)}
  $$
  can be identified with
  $$
  \xymatrix @C=48pt{
  B_{-1} \ar[r]^{s_{\bot-1}s_{\top+1}} & B_1 ,
  }
  $$
  which is a monomorphism since $B$ is complete.  It follows that the top map
  in the above pullback square is a monomorphism, as asserted.
  (Note the importance of completeness.)
\end{proof}

\section{The decomposition space of intervals}

\begin{blanko}{Interval category as a full subcategory in $\cDecomp$.}\label{towards-Int}
  We now invoke the general results about extension of factorisation 
  systems (Lemmas~\ref{fact-ext} and \ref{fact-Kl-E}).
  Let 
  $$
  \Int := \wtil \aInt
  $$
  denote the image of $\aInt\subset \cFD$ under the left adjoint
  $i\upperstar$ in
  the adjunction
  $$\xymatrix{
  i\upperstar  : \cFD \ar@<+3pt>[r] & \ar@<+3pt>[l]\cDecomp : u\upperstar 
  }$$
  as in~\ref{fact-Kl-E}.  Hence $\Int \subset \cDecomp$ is the full 
  subcategory of decomposition spaces underlying algebraic intervals.
  Say a map in $\Int $ is {\em stretched} if it is the $i\upperstar $ image of a 
  map in $\aInt$ (i.e.~a stretched map in $\cFD$).
\end{blanko}
  
\begin{prop}\label{fact-Int}
  The stretched maps as left-hand class and the \culf maps as right-hand class form a 
  factorisation system on $\Int$.
\end{prop}
\begin{proof}
  The stretched-cartesian factorisation system on $\cFD$ is compatible with the
  adjunction $i\upperstar \isleftadjointto u\upperstar$ and the subcategory 
  $\Int$ precisely as required to apply the general 
  Lemma~\ref{fact-Kl-E}. Namely, we have:

  --- $u\upperstar i\upperstar$
  preserves cartesian maps by Corollary~\ref{u*i*(cart)=cart}.

  ---  $u\upperstar \varepsilon$ is cartesian by \ref{u*f(cULF)=cart}, since
  $\varepsilon$ is \culf by \ref{counit-cULF}.

  --- If $A \to B$ is stretched, $A$ an algebraic interval, then so is $B$,
  by \ref{reduced/wide}.
  
  Lemma~\ref{fact-Kl-E} now tells us that there is a
  factorisation system on $\Int$ where the left-hand class are the maps of
  the form $i\upperstar $ of a stretched map.  The right-hand class of $\Int$,
  described by 
  Lemma~\ref{fact-Kl-E} as those maps $f$ for which $u\upperstar f$ is 
  cartesian,
  is seen by Lemma~\ref{cart&cULF} to be precisely the \culf maps.
\end{proof}

As in \ref{fact-Kl-E}, we can further restrict to the image of
the category $\Xiplus \subset \aInt$ consisting of
the representables together with the terminal object $\Xi[-2]$:

\begin{lemma}\label{IntDelta}
  The restriction (as in \ref{fact-Kl-E})
  to $\Xiplus \subset \aInt$ gives $\simplexcategory\subset \Int$:
  $$\xymatrix{
\simplexcategory \ar[r]^-{\mathrm{f.f.}} & \Int \ar[r]^-{\mathrm{f.f.}} & \cDecomp 
\ar@<+3pt>[d]^{u\upperstar}\\
\Xiplus  \ar[u] \ar[r] &\aInt \ar[u] \ar[r] & \cFD ,
\ar@<+3pt>[u]^{i\upperstar }
}$$ and the
  stretched-\culf factorisation systems on $\Int$ restricts to the
  active-inert factorisation system on $\simplexcategory$.
\end{lemma}

\begin{proof}
  By construction the objects are $[-2], [-1], [0], [1],\ldots$
  and the mapping
  spaces are
  \begin{eqnarray*}
    \Map_{\Int}(\Xi[k],\Xi[n]) & \simeq & \Map_\Decomp( i\upperstar \Xi[k], i\upperstar \Xi[n])\\
     & \simeq &\Map_{\widehat\simplexcategory}(\Delta[k+2],\Delta[n+2]) \\
     & \simeq & \simplexcategory([k+2],[n+2]).
  \end{eqnarray*}
  It is clear by the explicit description of $i\upperstar $ that it takes
  the maps in $\Xiplus $ to the active maps in $\simplexcategory$.  On the other hand,
  it is clear that the \culf maps in $\simplexcategory$ are the inert maps.
\end{proof}

\begin{blanko}{Arrow $\infty$-category and restriction to $\simplexcategory$.}
  Let $\Ar^{\sker}(\Int) \subset \Ar(\Int)$ denote the full subcategory of the arrow 
  $\infty$-category spanned by the stretched maps.  Recall (from \ref{ArED}) that 
  $\Ar^{\sker}(\Int)$ is a
  cartesian fibration over $\Int$ via the domain projection.
  We now restrict this cartesian fibration to $\simplexcategory \subset \Int$ as in 
  \ref{setup-end}:
  $$\xymatrix{
     \Ar^{\sker}(\Int)_{|\simplexcategory} \drpullback \ar[r]^-{\mathrm{f.f.}}\ar[d]_{\mathrm{dom}} & 
     \Ar^{\sker}(\Int) \ar[d]^{\mathrm{dom}} \\
     \simplexcategory \ar[r]_-{\mathrm{f.f.}} & \Int
  }$$
  We put
  $$
  \mathcal U := \Ar^{\sker}(\Int)_{|\simplexcategory} .
  $$
  $\mathcal U \to \simplexcategory$ is the {\em Cartesian fibration of subdivided 
  intervals}: the objects of $\mathcal U$ are the stretched 
  interval maps $\Delta[k] \to A$, which we think of as 
  subdivided intervals.
  The arrows are
  commutative squares
  $$\xymatrix{
  \Delta[k] \ar[r] \ar[d] & \Delta[n] \ar[d] \\
  A \ar[r] & B
  }$$
  where the downwards maps are stretched, and the
  rightwards maps are in $\simplexcategory$ and in $\cDecomp$, respectively.
  (These cannot be realised in the world of $\Xi\op$-spaces, and the necessity
  of having them was the whole motivation for constructing $\Int$.)
  By \ref{ArED}, the cartesian maps are squares
  $$\xymatrix{
  \Delta[k] \ar[r] \ar[d] & \Delta[n] \ar[d] \\
  A \ar[r]_{\culfsymb} & B   .
  }$$
  Hence, cartesian lifts are performed by precomposing and then coreflecting
  (i.e.~stretched-\culf factorising and keeping only the stretched part).
  For a fixed domain $\Delta[k]$, we have (in virtue of 
  Proposition~\ref{aInt-faithful})
  $$
  \Int^{\sker}_{\Delta[k]/} \simeq \aInt_{\Xi[k-2]/} .
  $$
\end{blanko}

\begin{blanko}{The (large) decomposition space of intervals.}\label{U}
  The cartesian fibration $\mathcal U =\Ar^{\sker}(\Int)_{|\simplexcategory} \to
  \simplexcategory$ determines a right fibration, $\mathcal U^\mathrm{cart} =
  \Ar^{\sker}(\Int)_{\mid \simplexcategory}^{\mathrm{cart}} \to \simplexcategory$, and
  hence by straightening (\cite{Lurie:HTT}, Ch.2) a simplicial
  $\infty$-groupoid
  $$
  \UU:\simplexcategory\op\to\widehat{\Grpd},
  $$ 
  where $\widehat{\Grpd}$ is the very large $\infty$-category of not
  necessarily small $\infty$-groupoids.  We shall see that it is a complete
  decomposition space.

  We shall postpone the straightening as long as possible, as it
  is more convenient to work directly with the right fibration $\mathcal U^\mathrm{cart} \to
  \simplexcategory$.
  Its fibre over $[k]\in\simplexcategory$ is the $\infty$-groupoid $\UU_k$ of
  $k$-subdivided intervals.  That is, an interval $\un A$ equipped with a
  stretched map $\Delta[k] \to A$.  Note that $\UU_1$ is equivalent to the
  $\infty$-groupoid $\Int^\eq$.  Similarly, $\UU_2$ is equivalent to the
  $\infty$-groupoid of subdivided intervals, more precisely intervals with
  a stretched map from $\Delta[2]$.  Somewhat more exotic is $\UU_0$, the
  $\infty$-groupoid of intervals with a stretched map from $\Delta[0]$.  This
  means that the endpoints must coincide.  This does not imply that the
  interval is trivial.  For example, any $\infty$-category with a zero
  object provides an example of an object in $\UU_0$.
\end{blanko}

\begin{blanko}{A remark on size.}
  The fibres of the right fibration $\mathcal U^\mathrm{cart} \to \simplexcategory$ are large
  $\infty$-groupoids.  Indeed, they are all variations of $\UU_1$, the
  $\infty$-groupoid of intervals, which is of the same size as the
  $\infty$-category of simplicial spaces, which is of the same size as
  $\Grpd$.  Accordingly, the corresponding presheaf takes values in
  large $\infty$-groupoids, and $\UU$ is therefore a large decomposition space.
  These technicalities do not affect the following results, but will
  play a role from \ref{conjecture} and onwards.
\end{blanko}

Among the active maps in $\UU$, in each degree the unique map $g:\UU_r \to \UU_1$
consists in forgetting the subdivision.  The space $\UU$ also has the codomain
projection $\UU \to \Int$.  In particular we can describe the $g$-fibre over a
given interval $A$:
\begin{lemma}\label{lem:U_r_A}
  We have a pullback square
  $$\xymatrix{
      (A_r)_a \drpullback \ar[r]\ar[d] &  \UU_r \ar[d]^g \\
     {*} \ar[r]_{\name A} & \UU_1
  }$$ 
  where $a\in A_1$ denotes the longest edge.
\end{lemma}
\begin{proof}
  Indeed, the fibre over a coslice is the mapping space, so the pullback is
at first
$$
\Map_{\mathrm{stretched}}(\Delta[r], A)
$$
But that's the full subgroupoid inside $\Map(\Delta[r], A) \simeq A_r$
consisting of the stretched maps, but that means those whose 
restriction to the long edge is $a$.
\end{proof}

\begin{thm}\label{UcompleteDecomp}
  The simplicial space $\UU: \simplexcategory\op \to \widehat{\Grpd}$ is a (large) complete
  decomposition space.
\end{thm}

\begin{proof}
  We first show it is a decomposition space.
  We need to show that for an active-inert  pullback square in $\simplexcategory\op$,
  the image under $\UU$ is a pullback:
  $$
  \xymatrix{\UU_k \drpullback\ar[r]^{f'} \ar[d]_{g'} & \UU_m \ar[d]^g \\
  \UU_n \ar[r]_f & \UU_s}$$
  This square is the outer rectangle in
   $$
  \xymatrix{\Int_{\Delta[k]/}^{\sker} \ar[r]^j \ar[d]_{g'} 
& \Int_{\Delta[k]/} \ar[r]^{f'} \ar[d]_{g'} 
& 
\Int_{\Delta[m]/}\ar[d]^g \ar[r]^w  & \Int_{\Delta[m]/}^{\sker}\ar[d]^g \\
  \Int_{\Delta[n]/}^{\sker} \ar[r]_j & 
  \Int_{\Delta[n]/}{} \ar[r]_f & 
\Int_{\Delta[s]/}\ar[r]_w & \Int_{\Delta[s]/}^{\sker}
  }$$
  (Here we have omitted taking maximal $\infty$-groupoids, but it doesn't affect the 
  argument.)
  The first two squares consist in precomposing with the inert 
  maps $f$, $f'$.  The result will no longer be a
  stretched map, so in the middle columns we allow arbitrary maps.
  But the final step just applies the coreflection to take the stretched part.  Indeed this is how cartesian lifting goes in
  $\Ar^{\sker}(\Int)$.
  The first square is a pullback since $j$ is fully
  faithful.  The last square is a pullback since it is a special case of
  Lemma~\ref{lem:Dx'Dx}.  The main point is the second square
  which is a pullback by Lemma~\ref{preservespullback} --- this is where we use that
  the active-inert square in $\simplexcategory\op$ is a pullback.
  
  To establish that $\UU$ is complete, we need to check
  that the map $\UU_0 \to \UU_1$ is a monomorphism.  This map is just the forgetful functor
  $$
  (\Int^{\sker}_{*/})^\eq \to \Int^\eq .
  $$
  The claim is that its fibres are empty or contractible.
  The fibre over an interval $\un A= i\upperstar  A$ is
  $$
  \Map_{\mathrm{stretched}}(\terminal,\un A)  \simeq \Map_{\aInt}( \Xi[-2], A) \simeq \Map_{\widehat 
  \Xi}(\Xi[-2],A).
  $$
  Note that in spite of the notation, $\Xi[-2]$ is not a representable:
  it is the terminal object, and it is hence the colimit of all the 
  representables.  It follows that $\Map_{\widehat \Xi}(\Xi[-2],A) = \lim A$.
  This is the limit of a cosimplicial diagram
  $$
  \lim A \stackrel e \longrightarrow * \rightrightarrows A_0 \cdots
  $$
  In general the limiting map of a cosimplicial diagram 
  does not have to be a monomorphism, but in this case
  it is, as all the coface maps (these are the degeneracy maps of $A$)
  are monomorphisms by completeness of $A$, and since $A_{-1}$ is contractible.
  Since finally $e$ is a monomorphism into the contractible space
  $A_{-1}$, the limit must be empty or contractible.  Hence $\UU_0 \to \UU_1$ is a 
  monomorphism,
  and therefore $\UU$ is complete.
\end{proof}

\section{Universal property of $\UU$}

The refinements discussed in \ref{fact-decomp} now pay off to give us the
following main result.  Let $\Int\comma\cDecomp$ denote the comma $\infty$-category (as
in \ref{fact-theorem}).  It is the full subcategory in $\Ar(\cDecomp)$ spanned
by the maps whose domain is in $\Int$.  Let $\Ar^{\sker}(\Int)$ denote the full
subcategory of $\Ar(\Int)$ spanned by the stretched maps.  Recall (from \ref{ArED})
that both $\Int\comma\cDecomp$ and $\Ar^{\sker}(\Int)$ are cartesian fibrations over
$\Int$ via the domain projections, and that the inclusion $\Ar^{\sker}(\Int) \to
\Int\comma\cDecomp$ commutes with the projections (but does not preserve
cartesian arrows).

\begin{theorem}\label{Thm:I}
  The inclusion functor $\Ar^{\sker}(\Int) \into \Int\comma\cDecomp$
  has a right adjoint
  $$I:\Int\comma\cDecomp \to \Ar^{\sker}(\Int)  ,
  $$
  which takes cartesian arrows to cartesian arrows.
\end{theorem}
\begin{proof}
  We have already checked, in the proof of \ref{fact-Int},
  that the conditions of the general Theorem~\ref{fact-theorem}
  are satisfied by the adjunction $i\upperstar \isleftadjointto u\upperstar $
  and
  the stretched-cartesian factorisation system on $\cFD$.  It remains to restrict
  this adjunction to the full subcategory 
  $\aInt \subset \cFD$.
\end{proof}
Note that over an interval $\un A$, the adjunction restricts to the adjunction
of \ref{I=LvwR} as follows:
$$\xymatrix{
   \Int^{\sker}_{\un A/} \ar@<+2.5pt>[r]\ar[d]_\simeq
   & \ar@<+2.5pt>[l]^-I \cDecomp_{\un A/} \ar@<+2.5pt>[d]^R \\
   \aInt_{A/} \ar@<+2.5pt>[r]^v &\ar@<+2.5pt>[l]^w \cFD_{A/} \ar@<+2.5pt>[u]^L
}$$ 

\bigskip

We now restrict these cartesian fibrations further to $\simplexcategory \subset \Int$.
We call the coreflection $I$, as it is the factorisation-interval construction:
$$\xymatrix{
\mathcal U=\Ar^{\sker}(\Int)_{|\simplexcategory}\drto_{\mathrm{dom}}
\ar@<+3pt>[rr] && \ar@<+3pt>[ll]^I
\simplexcategory\comma\cDecomp\dlto^{\mathrm{dom}}\\
&\simplexcategory
}
$$

The coreflection 
$$
I : \simplexcategory\comma\cDecomp  \to \mathcal U
$$
is a morphism of cartesian fibrations over $\simplexcategory$ (i.e.~preserves
cartesian arrows).  Hence it induces a morphism of right fibrations
$
I : (\simplexcategory\comma\cDecomp)^\mathrm{cart}  \to  \mathcal U^\mathrm{cart}.
$
\begin{theorem}\label{thm:IU=cULF}
  The morphism of right fibrations
  $$
  I : (\simplexcategory\comma\cDecomp)^\mathrm{cart}  \to  \mathcal U^\mathrm{cart}
  $$ 
  is \culf.
\end{theorem}

\begin{proof}
  We need to establish that,
  for $g: \Delta[k] \to \Delta[1]$ the unique active map in degree $k$,
  the following square is a pullback:
  $$\xymatrix{
     {\cDecomp_{\Delta[k]/}}^\eq \ar[r]^{\mathrm{pre. } g}\ar[d]_{I_k} & 
     {\cDecomp_{\Delta[1]/}}^\eq  \ar[d]^{I_1} \\
     {\Int^{\sker}_{\Delta[k]/}}^\eq \ar[r]_{\mathrm{pre. }g} & 
	 {\Int^{\sker}_{\Delta[1]/}}^\eq .
  }$$
  Here the functors $I_1$ and $I_k$ are the coreflections of Theorem~\ref{Thm:I}.
  We compute the fibres of the horizontal maps over a point $a: \Delta[1]\to 
  X$.  For the first row, the fibre is
  $$
  \Map_{\cDecomp_{\Delta[1]/}}(g,a) .
  $$
  For the second row, the fibre is
  $$
  \Map_{\Int^{\sker}_{\Delta[1]/}}(g, I_1(a)). 
  $$
  But these two spaces are equivalent by the adjunction of 
  Theorem~\ref{Thm:I}.
\end{proof}

Inside $\simplexcategory\comma\cDecomp$, we have the fibre over $X$, for the codomain
fibration (which is a cocartesian fibration).  This fibre is just 
$\simplexcategory_{/X}$,
the Grothendieck construction of the presheaf $X$.
This fibre clearly includes into the
cartesian part of $\simplexcategory\comma\cDecomp$.  

\begin{lemma}
  The associated morphism of right fibrations
  $$
  \simplexcategory_{/X} \to (\simplexcategory\comma\cDecomp)^\mathrm{cart}
  $$
  is \culf.
\end{lemma}

\begin{proof}
  For $g: \Delta[k] \to \Delta[1]$ the unique active map in degree $k$,
  consider the diagram
  $$\xymatrix{
   \Map(\Delta[k],X) \drpullback \ar[r]^-{\mathrm{pre. } g}\ar[d] & 
   \Map(\Delta[1],X)\drpullback \ar[d] 
   \ar[r] & \terminal \ar[d]^{\name X}\\
   {\cDecomp_{\Delta[k]/}}^\eq \ar[r]_-{\mathrm{pre. } g}& 
   {\cDecomp_{\Delta[1]/}}^\eq \ar[r]_-{\mathrm{codom}} & \cDecomp^\eq .}
   $$
  The right-hand square and the outer rectangle are obviously pullbacks, 
  as the fibres of coslices are the mapping spaces.  Hence the left-hand
  square is a pullback, which is precisely to say that the vertical map is
  \culf.  
\end{proof}

So altogether we have \culf map
$$
\simplexcategory_{/X} \to (\simplexcategory\comma\cDecomp)^\mathrm{cart} \to \mathcal U^\mathrm{cart},
$$
or, by straightening, a \culf map of complete decomposition spaces
$$
I: X \to \UU ,
$$
the {\em classifying map}.
It takes a $k$-simplex in $X$ to a $k$-subdivided interval,
as already detailed in Section~\ref{sec:fact-int}.

\medskip

The following conjecture expresses the idea that $\UU$ should be terminal in the 
$\infty$-category of complete decomposition spaces and \culf maps,
but since $\UU$ is large this cannot literally be true,
and we have to formulate it slightly differently.
\begin{blanko}{Conjecture.}\label{conjecture}
  {\em $\UU$ is the universal complete decomposition space for \culf maps.
  That is, for each (legitimate) complete decomposition space $X$, the space 
  $\Map_{\cDecomp^{\culfsymb}}(X,\UU)$ is contractible.
}
\end{blanko}
At the moment we are only able to prove the following weaker statement.

\begin{theorem}\label{thm:connected}
  For each (legitimate) complete decomposition space $X$, the space 
  $\Map_{\cDecomp^{\culfsymb}}
(X,\UU)$ is connected.
\end{theorem}
\begin{proof}
  Suppose $J: X\to \UU$ and $J': X \to \UU$ are two \culf functors.  
  As in the proof of Theorem~\ref{thm:IU=cULF}, \culf{}ness is equivalent to saying
  that we have a pullback
    $$\xymatrix{
     \Map_{\cDecomp}(\Delta[k],X) \drpullback \ar[r]^{\mathrm{pre. } g}\ar[d]_{J_k} & 
     \Map_{\cDecomp}(\Delta[1],X) \ar[d]^{J_1} \\
     {\Int^{\sker}_{\Delta[k]/}}^\eq \ar[r]_{\mathrm{pre. }g} & 
     {\Int^{\sker}_{\Delta[1]/}}^\eq  .
  }$$
  We therefore have equivalences between the fibres over
  a point $a: \Delta[1]\to 
  X$:
  $$
  \Map_{\cDecomp_{\Delta[1]/}}(g,a) \simeq
   \Map_{\Int^{\sker}_{\Delta[1]/}}(g, J_1(a)).
   $$
  But the second space is equivalent to $\Map_{\Int^{\sker}}(\Delta[k],J_1(a))$.
  Since these equivalences hold also for $J'$, we get
  $$
  \Map_{\Int^{\sker}}(\Delta[k],J_1(a)) \simeq 
  \Map_{\Int^{\sker}}(\Delta[k],J'_1(a)),
  $$
  naturally in $k$.  This is to say that $J_1(a)$ and $J_1'(a)$ are
  levelwise equivalent simplicial spaces.  But a \culf map is determined by its
  $1$-component, so $J$ and $J'$ are equivalent in the functor $\infty$-category.  In
  particular, every object in $\Map^{\culfsymb}(X,\UU)$ is equivalent to the canonical $I$
  constructed in the previous theorems. 
\end{proof}

\begin{blanko}{Size issues and cardinal bounds.}\label{kappa}
  We have observed that the decomposition space of intervals is large, in the
  sense that it takes values in the very large $\infty$-category of large
  $\infty$-groupoids.  This size issue prevents $\UU$ from being a terminal object in
  the $\infty$-category of decomposition spaces and \culf maps.

  A more refined analysis of the situation is possible by standard techniques,
  by imposing cardinal bounds, as we briefly explain.
  For $\kappa$ a regular uncountable cardinal, say that a simplicial space
  $X:\simplexcategory\op\to\Grpd$ is {\em $\kappa$-bounded}, when for each $n\in\simplexcategory$
  the space $X_n$ is $\kappa$-compact.  In other words, $X$ takes values in
  the (essentially small) $\infty$-category
  $\Grpd^\kappa$ of $\kappa$-compact $\infty$-groupoids.  Hence the $\infty$-category
  of $\kappa$-bounded simplicial spaces is essentially small.  The attribute
  $\kappa$-bounded now also applies to decomposition spaces and intervals.
  Hence the $\infty$-categories of $\kappa$-bounded decomposition spaces and
  $\kappa$-bounded intervals are essentially small.  Carrying the $\kappa$-bound
  through all the constructions, we see that there is an essentially small
  $\infty$-category $\UU_1$ of $\kappa$-bounded intervals, and a legitimate
  presheaf $\UU^\kappa: \simplexcategory\op\to\Grpd$ of $\kappa$-bounded intervals.

  It is clear that if $X$ is a $\kappa$-bounded decomposition space, then
  all its intervals are $\kappa$-bounded too.  It follows that if
  Conjecture~\ref{conjecture} is true then it is also true that $\UU^\kappa$,
  the (legitimate) decomposition space of all $\kappa$-bounded intervals,
  is universal for $\kappa$-bounded decomposition spaces, in the sense that
  for any $\kappa$-bounded decomposition space $X$, the space
  $\Map_{\cDecomp^{\culfsymb}}(X,\UU^\kappa)$ is contractible.
\end{blanko}

\section{\M intervals and the universal \M function}
\label{sec:MI}

We finally impose the \M condition.

\begin{blanko}{Nondegeneracy.}
  Recall from \cite[2.12]{GKT:DSIAMI-2}
  that for a complete decomposition space $X$ we have
  $$
  \nondeg X_r \subset X_r
  $$
  the full subgroupoid of $r$-simplices none of whose principal edges are 
  degenerate.   These can also be described as the full subgroupoid
  $$
  \nondeg X_r \simeq \Map_{\operatorname{nondegen}}(\Delta[r],X) \subset
  \Map(\Delta[r],X)\simeq X_r
  $$
  consisting of the nondegenerate maps.

  Now assume that $A$ is an interval.  Inside
  $$
  \Map_{\operatorname{nondegen}}(\Delta[r],A) \simeq \nondeg A_r
  $$
  we can further require the maps to be stretched.  It is clear that this 
  corresponds to considering only nondegenerate simplices whose
  longest edge is the longest
    edge $a \in A_1$:
  \begin{lemma}\label{wide+nondegA}
    $$
    \Map_{\operatorname{stretched+nondegen}}(\Delta[r],A) \simeq (\nondeg A_r)_a .
    $$
  \end{lemma}
\end{blanko}

\begin{blanko}{Nondegeneracy in $\UU$.}
  In the case of $\UU: \simplexcategory\op\to\widehat{\Grpd}$, it is easy to describe the
  spaces $\nondeg \UU_r$.  They consist of stretched maps $\Delta[r] \to A$ for which
  none of the restrictions to principal edges $\Delta[1] \to A$ are degenerate.
  In particular we can describe the fibre over a given interval $A$
  (in analogy with \ref{lem:U_r_A}):
\end{blanko}
\begin{lemma}\label{lem:nondegU_r_A}
  We have a pullback square
  $$\xymatrix{
     (\nondeg A_r)_a \ar[r]\ar[d] & \nondeg \UU_r \ar[d] \\
     {*} \ar[r]_{\name A} & \UU_1   .
  }$$ 
\end{lemma}

\begin{blanko}{Finiteness conditions and \M intervals.}\label{bla:tight}
  Recall from \cite{GKT:DSIAMI-2} that a decomposition space $X$ is called 
  {\em locally finite} when $X_1$ is locally finite, and both $s_0 : X_0 \to 
  X_1$ and $d_1: X_2 \to X_1$ are finite maps.  Recall also that a complete
  decomposition space $X$ is called of locally finite length, or just {\em 
  tight}, when for each $a\in X_1$, there is an upper bound on the dimension
  of simplices with long edge $a$.  Recall finally that a complete 
  decomposition space
  is  called \M when it is locally finite
  and of locally finite length (i.e.~tight).  The \M condition can also 
  be formulated by saying that $X_1$ is locally finite and
  the `long-edge' map $$\sum \nondeg X_r \to X_1$$ is finite.

  A {\em \M interval} is an interval which is \M as a decomposition space.
\end{blanko}

\begin{prop}\label{prop:Mint=Rezk}
  Any \M interval is a Rezk complete Segal space.
\end{prop}
\begin{proof}
  Just by being an interval it is a Segal space (by \ref{prop:i*flanked=Segal}).
  Now the filtration condition implies the Rezk condition by
  \cite[Corollary~8.7]{GKT:DSIAMI-2}.
\end{proof}

\begin{lemma}\label{XFILT=>IFILT}
  If $X$ is a \FILT decomposition space, then for each $a\in X_1$, the
  interval $I(a)$ is a \FILT decomposition space.
\end{lemma}
\begin{proof}
  We have a \culf map $I(a)\to X$, and anything \culf over tight is again tight
  (see \cite[Proposition 6.5]{GKT:DSIAMI-2}).
\end{proof}

\begin{lemma}
  If $X$ is a locally finite decomposition space then for each $a\in X_1$,
  the interval $I(a)$ is a locally finite decomposition space.
\end{lemma}

\begin{proof}
  The morphism of decomposition spaces $I(a) \to X$ was constructed by pullback
  of the map $\terminal \stackrel{\name a}\to X_1$ which is finite  since $X_1$ is locally finite
  (see \cite[Lemma 3.16]{GKT:HLA}). 
  Hence $I(a) \to X$
  is a finite morphism of decomposition spaces, and therefore $I(a)$ is locally finite 
  since $X$ is.
\end{proof}
From these two lemmas we get
\begin{cor}\label{XM=>IM}
  If $X$ is a \M decomposition space, then for each $a\in X_1$, the
  interval $I(a)$ is a \M interval.
\end{cor}

\begin{prop}\label{prop:Mint=finite}
  If $A$ is a \M interval then for every $r$, the space $A_r$ is finite.
\end{prop}

\begin{proof}
  The squares
  $$\xymatrix{
     A_0 \drpullback \ar[r]^-{s_{\top+1}}\ar[d] & A_1 \drpullback 
     \ar[r]^-{s_{\bot-1}}\ar[d]_{d_0} & A_2 \ar[d]^{d_1} \\
     \terminal \ar@/_1pc/[rr]_{\name a} \ar[r]^-{s_{\top+1}} & A_0\ar[r]^-{s_{\bot-1}} & A_1
  }$$
  are pullbacks by the flanking condition \ref{flanked} (the second is a bonus
  pullback, cf.~\ref{s-1-newbonuspullbacks}).  The bottom composite arrow picks
  out the long edge $a\in A_1$.  (That the outer square is a pullback can be
  interpreted as saying that the $2$-step factorisations of $a$ are parametrised
  by their midpoint, which can be any point in $A_0$.)  Since the active maps 
  of $A$ are finite (simply by the assumption that $A$ is locally finite) 
  in particular the map $d_1:A_2 \to A_1$ is finite, hence
  the fibre $A_0$ is finite.  The same argument works for arbitrary $r$, by
  replacing the top row by $A_r \to A_{r+1} \to A_{r+2}$, and letting the columns
  be $d_0^r$, $d_0^r$ and $d_1^r$.
\end{proof}
\noindent
(This can be seen as a homotopy version of \cite[Lemma~2.3]{LawvereMenniMR2720184}.)

\begin{cor}
  For a \M interval, the total space of all nondegenerate simplices $\sum_r \nondeg 
  A_r$ is finite.
\end{cor}
\begin{proof}
  This follows from Proposition~\ref{prop:Mint=finite} and the fact that
  a complete decomposition space is \M if and only if the map
  $$
  \sum_r d_1{}^{r-1}: \sum_r \nondeg X_r \to X_1$$
  is finite (see \ref{bla:tight}).  
\end{proof}

\begin{cor}\label{cor:<kappa}
  A \M interval is $\kappa$-bounded for any uncountable cardinal $\kappa$.
\end{cor}

\begin{blanko}{The decomposition space of \M intervals.}  
  There is a decomposition space $\MI\subset \UU$ consisting of all \M
  intervals.  In each degree, $\MI_k$ is the full subgroupoid of $\UU_k$
  consisting of the stretched maps $\Delta[k] \to A$ for which $A$ is \M. While $\UU$
  is large, $\MI$ is a legitimate decomposition space by \ref{cor:<kappa} and
  \ref{kappa}.
\end{blanko}

\begin{theorem}\label{thm:MI=M}
  The decomposition space $\MI$ is \M.
\end{theorem}
\begin{proof}
  We first prove that the map $\sum_r \nondeg \MI_r \to \MI_1$ is a finite 
  map. Just check the fibre: fix a \M interval $A\in \MI_1$, with longest edge
  $a\in A_1$.
  From Lemma~\ref{lem:nondegU_r_A} we see that the fibre over $A$ is 
  $(\sum_r \nondeg A_r)_a=\sum_r (\nondeg A_r)_a$.  But this is the fibre 
  over $a\in  A_1$ of the map $\sum_r \nondeg A_r \to A_1$, which is finite
  by the assumption that $A$ is \M.
  
  Next we show that the $\infty$-groupoid $\MI_1$ is locally finite.
  But $\MI_1$ is the space of
  \M intervals, a full subcategory of the space of all decomposition spaces,
  so we need to show, for any \M interval $A$, that $\Eq_{\Decomp}(A)$ is 
  finite.  Now we exploit an important property of \M decomposition spaces,
  namely that they are {\em split} \cite{GKT:DSIAMI-2}: this means that face maps preserve 
  nondegenerate simplices.  The key feature of split decomposition spaces is
  that they are essentially semi-decomposition spaces 
  (i.e.~$\Deltainj\op$-spaces satisfying the decomposition-space axioms for 
  face maps) with degeneracies freely added.  More formally, restriction along
  $\Deltainj \to \simplexcategory$ yields an equivalence of $\infty$-categories
  between split decomposition spaces and \culf maps, and semi-decomposition spaces 
  and ULF maps \cite[5.20]{GKT:DSIAMI-2}.
  
  Since $A$ is split, we can compute $\Eq_{\Decomp}(A)$ inside
  the $\infty$-groupoid of split decomposition spaces,
  which is equivalent to the $\infty$-groupoid
  of semi-decomposition spaces.  So we have reduced to computing
  $$
  \Map_{\Fun(\Deltainj\op,\Grpd)}(\nondeg A,\nondeg A) .
  $$
  Now we know that all $\nondeg A_k$ are finite, so the mapping space can 
  be computed in the functor $\infty$-category with values in $\grpd$, the 
  $\infty$-category of finite $\infty$-groupoids.
  On the other hand we also know that these
  $\infty$-groupoids are
  empty for $k$ big enough, say 
  $\nondeg A_k = \varnothing$ for $k > r$.  Hence we can compute this mapping 
  space as a functor $\infty$-category on the truncation $\Deltainj^{\leq r}$.
  So we are finally talking about a functor $\infty$-category over a finite simplicial set
  (finite in the sense: only finitely many nondegenerate simplices), and with 
  values in finite $\infty$-groupoids.
  So we are done by the following lemma.
\end{proof}

Recall that $\grpd$ denotes the $\infty$-category of finite $\infty$-groupoids.
\begin{lemma}
  Let $K$ be a finite simplicial set, and let $X$ and $Y$ be 
  presheaves on $K$
    valued in finite $\infty$-groupoids.
  Then
  $$
  \Map_{\Fun(K\op,\grpd)}(X,Y)
  $$
  is finite.
\end{lemma}
\begin{proof}
This mapping space may be calculated as the limit of the diagram
$$
\widetilde 
K\xrightarrow{\;f\;}K\times K^{\op}\xrightarrow{X^{\op}\times Y}
\grpd\op\times\grpd\xrightarrow{\Map}\Grpd .
$$
See for example \cite[Proposition 2.3]{glasman-1408.3065v3} for a proof. 
Here $\widetilde K$ is the edgewise subdivision of $K$, introduced 
in \cite[Appendix 1]{Segal:1973} as follows:
$$
\widetilde K_n=K_{2n+1},
\qquad \widetilde d_i=d_id_{2n+1-i},
\quad \widetilde s_i=s_is_{2n+1-i}  ,
$$ 
  and $f:\widetilde K\to K\times K\op$ is defined by
  $(d_{n+1},d_0)^{n+1}:K_{2n+1}\to K_n\times K_n$.  Now $\widetilde K$ is also
  finite: for each nondegenerate simplex $k$ of $K$, only a finite number of the
  degeneracies $s_{i_j}\dots s_{i_1}k$ will be nondegenerate in $\widetilde K$.
  Furthermore, mapping spaces between finite
  $\infty$-groupoids are again finite, since
  $\grpd$ is cartesian closed (see \cite[Proposition 3.17]{GKT:HLA}).
  Thus the mapping space in question can be computed as a finite limit of finite
  $\infty$-groupoids, so it is again finite
  (see~\cite[Proposition 3.9]{GKT:HLA}). 
\end{proof}

\begin{prop}
  Let $X$ be a decomposition space with locally finite $X_1$.  Then the 
  following are equivalent.
  \begin{enumerate}
    \item $X$ is \M.
  
    \item All the intervals in $X$ are \M.
  
    \item Its classifying map factors through  $\MI\subset \UU$.
  \end{enumerate}
\end{prop}

\begin{proof}
    If the classifying map factors through $X\to \MI$, then $X$ is \culf over a
    \M space, hence is itself \FILT (by
    \cite[Proposition 6.5]{GKT:DSIAMI-2}), and has finite active maps.
    Since we have assumed $X_1$ locally finite, altogether $X$ is \M. We already
    showed (\ref{XM=>IM}) that if $X$ is \M then so are all its intervals.
    Finally if all the intervals are \M, then clearly the classifying map
    factors through $\MI$.
\end{proof}

\begin{BM}
  For $1$-categories, Lawvere and Menni~\cite{LawvereMenniMR2720184}
  show that a category is \M if and only
  if all its intervals are \M. This is not quite true in our setting: even if
  all the intervals of $X$ are \M, and in particular finite, there is no
  guarantee that $X_1$ is locally finite.
\end{BM}

\begin{blanko}{Conjecture.}
  The decomposition space $\MI$ is terminal in the $\infty$-category of \M decomposition
  spaces and \culf maps.
\end{blanko}
This would follow from Conjecture \ref{conjecture}, but could be strictly weaker.

\begin{blanko}{\M functions.}
  Recall from \cite{GKT:DSIAMI-2} that for a complete decomposition space $X$,
  for each $k\geq 0$, we have the linear functor $\Phi_k$ defined by the span
  $X_1 \leftarrow \nondeg X_k \to \terminal$, and that these assemble into the \M
  function, namely the formal difference 
  $$\mu = \Phieven-\Phiodd,$$ which is
  convolution inverse to the zeta functor $\zeta$ given by the span
  $X_1 \leftarrow X_1 \to \terminal$.  Since we cannot directly make sense of the minus 
  sign, the actual \M inversion formula is expressed as a canonical equivalence
  of $\infty$-groupoids
  $$
  \zeta * \Phieven \simeq \varepsilon + \zeta * \Phiodd.
  $$
  When furthermore $X$ is
  a \M decomposition space, then this equivalence admits a cardinality 
  (see \cite[Theorem 8.9]{GKT:DSIAMI-2}),
  which is the \M inversion formula in $\Q$-vector spaces (where the minus
  sign {\em can} be interpreted).
\end{blanko}

\begin{blanko}{The universal \M function.}
  The decomposition space $\UU$ of all intervals is complete, hence it has \M
  inversion at the objective level as just described.
  Note that the map
  $m:\nondeg \UU_k\to \UU_1$ in $\widehat{\Grpd}$, that defines $\Phi_k$, has fibres
  in $\Grpd$ by Lemma~\ref{lem:U_r_A}.  
  Now it is a general fact that for a \culf map $f:X \to Y$ between complete 
  decomposition spaces, we have 
  $f\upperstar \Phi_k = \Phi_k$   (see \cite[3.9]{GKT:DSIAMI-2}).
  Since every complete decomposition space
  $X$ has a canonical \culf map to $\UU$, 
  it follows 
  that the \M function of $X$ is induced from that of $\UU$.  The latter can
  therefore be called the {\em universal \M function}.

  The same reasoning works in the \M situation, and implies the existence of a
  universal \M function numerically.  Namely, since $\MI$ is \M, its \M 
  inversion formula admits a cardinality.
\end{blanko}

\begin{theorem}
  In the incidence algebra $\Q^{\pi_0 \MI}$, the zeta function $\norm \zeta: 
  \pi_0 \MI \to \Q$ is invertible under convolution, and its inverse is
  the universal \M function
  $$
  \norm \mu := \norm{\Phieven}-\norm{\Phiodd} .
  $$
  The \M function in the (numerical) incidence algebra of any \M decomposition
  space is induced from this universal \M function via the classifying map.
\end{theorem}

\begin{blanko}{Comparison with Lawvere--Menni.}\label{LM-comparison}
  The idea of a universal Hopf algebra of \M intervals is due to Lawvere,
  and an objective construction of it was first given by
  Lawvere--Menni~\cite{LawvereMenniMR2720184}.  We comment on the
  differences between their setting and approach and ours.
  
  Our setting is the symmetric monoidal $\infty$-category
  $(\LIN,\tensor,\Grpd)$ whose objects are slices of $\Grpd$.  The slice
  $\Grpd_{/X}$ is the homotopy-sum completion of the $\infty$-groupoid $X$.
  Our coalgebras live in $\LIN$, and our convolution algebras live in the
  linear dual, whose objects are presheaf categories $\Grpd^X$ (also the
  homotopy-sum completion of the $\infty$-groupoid $X$).  This means our
  coefficients are $\infty$-groupoids.  (To be precise, finiteness
  conditions should be imposed, as we do.  For all details, see
  \cite{GKT:HLA}.)  The incidence coalgebra of \M intervals is thus the
  slice $\Grpd_{/\UU_1}$, where $\UU_1$ is the $\infty$-groupoid of \M
  intervals.  To arrive at ordinary algebra, we take just take homotopy 
  cardinality. Our `objectification' is thus the most simple-minded, replacing
  numbers by sets, groupoids, $\infty$-groupoids.
   
  Lawvere and Menni do not go in the homotopy direction of
  $\infty$-groupoids, but consider instead a somewhat more subtle
  objectification which involves a level of non-invertible arrows.  Their
  setting is the symmeric monoidal $2$-category $(\kat{EXT}, \tensor,
  \Set)$ of extensive categories, i.e.~categories $\EE$ for which the
  canonical functor $\EE/(A+B) \to \EE/A \times \EE/B$ is an
  equivalence~\cite{Carboni-Lack-Walters}.  Their objective coalgebra is
  the extensive category $\operatorname{Fam}(\kat{sM\"oI})$, the finite-sum
  completion of the category of \M intervals and their stretched maps (in our
  terminology), and the convolution algebra is the extensive dual,
  $\kat{Cat}(\kat{sM\"oI}, \Set)$.  To arrive at ordinary algebra, they
  apply the Burnside-rig construction for extensive categories, which
  amounts to taking isomorphism classes.  To bypass the $2$-categorical
  hassle of $(\kat{EXT}, \tensor, \Set)$, they actually work with {\em
  extensive procomonoids} rather than comonoids.  In the case at hand, this
  means a functor $\Delta: \kat{sM\"oI} \to
  \operatorname{Fam}(\kat{sM\"oI}\times \kat{sM\"oI})$.  Here the
  $\operatorname{Fam}$ on the right is what allows to write formal sums.
  This induces the monoidal structure on $\kat{Cat}(\kat{sM\"oI}, \Set)$ by
  Day convolution.
   
  An important thing to note in the Lawvere--Menni set-up is that
  $\kat{sM\"oI}$ is not just a groupoid, and that their construction is
  therefore functorial also in some maps that are not invertible, namely
  the stretched interval maps.  The full significance of this extra functoriality is
  not clear to us.  It is invisible at the algebraic level.  
  Regarding the universal property, note that the \M categories for which it is 
  supposed to be universal do {\em not} have non-invertible 
  interval maps: the varying incidence coalgebras
  are of the form $\operatorname{Fam}(X_1)$, where $X_1$
  is the set of arrows of a \M category, and in particular is discrete.
  
  The non-invertible aspect is only implicit in our construction.  Namely,
  the universal decomposition space $\UU$ was constructed by taking the right
  fibration associated to $\mathcal{U}\to\simplexcategory$, which in turn involves
  the stretched maps.  The reason for discarding these maps was just to get a
  decomposition {\em space}, not a simplicial $\infty$-category, 
  so as to be able to take cardinality and relate to classical 
  theory.  We leave the development of a theory of {\em decomposition categories}
  for another occasion.
\end{blanko}

\end{document}